\documentclass[12pt,a4paper]{article}
\usepackage[utf8]{inputenc}
\usepackage[T1]{fontenc}
\usepackage{amsmath}
\usepackage{amsthm}
\usepackage{amsfonts}
\usepackage{amssymb}
\usepackage{dsfont}
\usepackage{graphicx}
\usepackage{tikz}
\usepackage{stmaryrd}
\usepackage{xcolor}
\usepackage{caption}
\usepackage{hyperref}
\usepackage[body={16cm, 22cm},centering]{geometry}

\newtheorem{proposition}{Proposition}[section]
\newtheorem{lemma}{Lemma}[section]
\newtheorem{theorem}{Theorem}[section]

\theoremstyle{definition}
\newtheorem{definition}{Definition}[section]
\newtheorem{remark}{Remark}[section]

\newtheorem{example}{Example}[section]

\newtheorem*{solution*}{Solution}

\DeclareMathOperator{\Mat}{Mat}
\DeclareMathOperator{\ran}{ran}

 \newcommand{\bbar}[1]{\setbox0=\hbox{$#1$}\dimen0=.2\ht0 \kern\dimen0 \overline{\kern-\dimen0 #1}}

 \DeclareMathOperator{\End}{\ensuremath{\mathcal{E}\kern-.125em\mathpzc{nd}}}

 \DeclareMathOperator{\Hom}{\mathcal{H}\kern-.125em\mathpzc{om}}
 
 \DeclareMathOperator{\id}{id}
 \DeclareMathOperator{\im}{im}

 \renewcommand{\setminus}{\smallsetminus}

  \DeclareMathOperator{\ind}{ind}

  \DeclareMathOperator{\sym}{Sym}

 \newcommand{\udot}{\ensuremath{{\lower .183333em \hbox{\LARGE \kern -.05em$\cdot$}}}}

 \DeclareMathOperator{\Hess}	{Hess} 
 \DeclareMathOperator{\Sym}	{Sym}

    \DeclareMathOperator{\rank}	{rank}  
        \DeclareMathOperator{\codim}	{codim} 
   
 \newcommand{\cL}{\mathcal{L}}
 \newcommand{\cM}{\mathcal{M}}

 \newcommand{\cU}{\mathcal{U}} 
 \newcommand{\cV}{\mathcal{V}}

 \renewcommand{\L}{\mathbb{L}}
 
 \newcommand{\N}{\mathbb{N}}

 \newcommand{\R}{\mathbb{R}}

  \newcommand{\Z}{\mathbb{Z}}

 \newcommand{\p}{\partial}

 \DeclareMathOperator{\spn}{span}
 \DeclareMathOperator{\Mi}{Mi}
 \DeclareMathOperator{\Li}{Li}
 \DeclareMathOperator{\Ki}{Ki}
 \DeclareMathOperator{\sign}{sign}

\newcommand{\footremember}[2]{%
    \footnote{#2}
    \newcounter{#1}
    \setcounter{#1}{\value{footnote}}%
}

\title{Jacobi Fields in Optimal Control: Morse and Maslov Indices}
\author{%
  Andrei Agrachev\footremember{alley}{Scuola Internazionale Superiore di Studi Avanzati, Trieste, Italy, agrachevaa@gmail.com}%
  \and Ivan Beschastnyi\footremember{trailer}{Scuola Internazionale Superiore di Studi Avanzati, Trieste, Italy, i.beschastnyi@gmail.com}%
  }
\date{}

\begin{document}

\maketitle 
 
\begin{abstract}In this paper we discuss a general framework based on symplectic geometry for the study of second order conditions in constrained variational problems on curves. Using the notion of $L$-derivatives we construct Jacobi curves, which represent a generalization of Jacobi fields from the classical calculus of variations, but which also works for non-smooth extremals. This construction includes in particular the previously known constructions for specific types of extremals. We state and prove Morse-type theorems that connect the negative inertia index of the Hessian of the problem to some symplectic invariants of Jacobi curves.
\end{abstract}

\section*{Introduction}
\label{sec:motiv}
Consider a standard variational problem of the form
\begin{equation}
\label{eq:lagr_exam}
J[q(t)] = \int_0^T L(q,\dot q)dt \to \min, \qquad q\in \R^n,
\end{equation}
with fixed boundary conditions $q(0) = q_0$, $q(T) = q_T$. We assume that our Lagrangian is sufficiently smooth and that the strong Legendre condition holds
$$
L_{\dot{q}\dot{q}} > 0.
$$
We also use the usual convention of summation over the repeating indices whenever it is possible. Given a map $F: H_1 \to H_2$ between two Banach manifolds we denote by $dF[u](v)$ the differential of $F$ at a point $u\in H_1$ applied to a vector $v\in T_u H_1$. Sometimes, we will use a shorter notation $d_u F$ for brevity.

If $\tilde{q}(t)$ is a minimum, then it must be a critical point of $J$, i.e.
$$
d J[\tilde q(t)](\phi(t)) =0, \qquad \forall \phi\in C^2 \,:\,\phi(0) = \phi(T) = 0.
$$
Equivalently one can say that $\tilde{q}(t)$ satisfies the Euler-Lagrange equation which can be written in a Hamiltonian form as
\begin{align}
\dot q^i &= \frac{\p H}{\p p_i},\nonumber \\
\dot p_i &= -\frac{\p H}{\p q^i}\label{eq:ham_example}
\end{align}
with a Hamiltonian
$$
H =  p_i \dot{q}^i - L(q,\dot{q}), \qquad p_i = L_{\dot{q}^i}.
$$
The Legendre condition ensures that we can resolve at least locally $p_i = L_{\dot{q}^i}$ with respect to $\dot{q}$ and obtain a Hamiltonian $H(p,q)$ that depends only on the phase variables.

After that, one usually proceeds to studying the necessary second order conditions like $d^2 J[\tilde q(t)]\geq 0$. But one can ask a more general question of calculating the \textit{negative inertia index} $\ind^- d^2 J[\tilde q(t)]$. It can be done via the notion of conjugate points due to Jacobi. Namely we linearize system (\ref{eq:ham_example}) to get the Jacobi equation
\begin{align}
\dot x^i &= \frac{\p^2 H}{\p p_i \p q^j}x^j + \frac{\p^2 H}{\p p_i \p  p_j}y_j\nonumber, \\
\dot y_i &= -\frac{\p^2 H}{\p q^i \p q^j} x^j -\frac{\p^2 H}{\p q^i \p p_j} y_j\label{eq:jac_example}.
\end{align}
A moment of time $t_{c}$ is called \emph{conjugate} if there exists a non-trivial solution of the Jacobi equation (\ref{eq:jac_example}) with boundary conditions $x(0) = 0$, $x(t_{c}) = 0 $. The corresponding point $\tilde q(t_{c})$ of our extremal is called a \emph{conjugate point} and the number of linearly independent solutions of the Jacobi boundary value problem is called the \emph{multiplicity} of $\tilde q(t_{c})$. 

A famous theorem due to Morse states
\begin{theorem}[Morse]
The index $\ind^- d^2 J[\tilde q(t)]$ is equal to the number of conjugate points counted with multiplicities.
\end{theorem}
Due to the strong Legendre condition conjugate points can not accumulate and the index must be finite.

We can give a more geometric interpretation of conjugate points using symplectic geometry. Consider the standard symplectic form on $\R^{2n}$
$$
\sigma\left( (x_1,y_1),(x_2,y_2) \right) = \begin{pmatrix}
x_1^T & y_1^T
\end{pmatrix}\begin{pmatrix}
0 & \id \\
-\id & 0 
\end{pmatrix}
\begin{pmatrix}
x_2 \\
y_2
\end{pmatrix}
$$
A subspace $L \in \R^{2n}$ is called \emph{isotropic} if the restriction $\sigma|_L$ is zero. A \emph{Lagrangian} subspace or a "Lagrangian plane" $L$ is the maximal isotropic subspace, which means that $\sigma|_L = 0$ and $\dim L = n$. For example, it is easy to see that the vertical subspace $\Pi = \{(0,y)\in \R^{2n}\}$ is Lagrangian. The set of Lagrangian planes is called the \emph{Lagrangian Grassmanian}.

Notice that the Jacobi equation (\ref{eq:jac_example}) is a Hamiltonian system on $\R^{2n}$. So its flow $\Phi_t$ is a linear symplectic transformation and hence it maps Lagrangian planes to Lagrangian planes. Thus by fixing an initial Lagrangian plane $L_0 = \Pi$ we get a curve $L_t = \Phi_t (\Pi)$ in the Lagrangian Grassmanian that is known as the \emph{Jacobi curve}. Then a moment of time $t$ is conjugate if, and only if, $L_t \cap \Pi \neq \{0\}$ and the multiplicity of the corresponding conjugate point is given by $\dim (L_t \cap \Pi)$.

The set of all Lagrangian planes that have a non zero intersection with a fixed Lagrangian plane $\Pi$ is called the \emph{Maslov train} $\cM_\Pi$. We have a conjugate point whenever our curve $L_t$ crosses it. The set $\cM_\Pi$ is an algebraic hypersurface in the Lagrange Grassmanian with codimension three singularities and a coorientation. Therefore there is a well defined intersection index with curves, which is called the \emph{Maslov index}. We will give precise definitions later, see also~\cite{gosson,arnoldgivental,sternberg}. Thus we can reformulate the Morse index theorem as follows.
\begin{theorem}
If $q(T)$ is not a conjugate point, then $\ind^- d^2 J[\tilde q(t)]$ is equal to the Maslov index of the curve $L_t |_{[\varepsilon,T]}$ for some $\varepsilon$ small enough.
\end{theorem}

So we see that a study of the functional $J[q]$ on a infinite dimensional space can be reduced to the study of behaviour of some curves in finite dimensional spaces. In the case of the first variation to the study of the curves defined as solutions of the Hamiltonian system (\ref{eq:ham_example}) and in the case of the second variation to the study of the Jacobi curve $L_t$. Moreover the Maslov index turns out to be a very flexible tool for computing the index of Hessians in variational problems (see for example~\cite{duistermaat,sachkov_elastica}), due to its homotopy invariance.

A natural question is whether or not this theory generalizes to constrained variational problems, like the problems of optimal control. In the case of the first variation the answer was given by Pontryagin and his students, who showed, that under very weak conditions minimal curves satisfy a Hamiltonian system with a Hamiltonian defined by a maximum condition, that is now widely known as the Pontryagin's maximum principle (PMP)~\cite{pmp}. 

To motivate such a generalization let us look at some possible applications of this theory to engineering and pure mathematics. First of all, using those techniques it is possible to answer certain stability questions. Indeed, many physical problems are formulated using variational principles. For example, given an elastic rod, what shapes can it take? One knows that those curves must be local minimizers of the bending energy and the overall problem can be formalized as a constrained variational problem~\cite{sachkov_elastica}. The PMP allows us to characterize critical curves, but most of them are not going to give a local minima of the functional. This means that even using infinitesimal perturbation of a given solution, one can produce curves with a smaller value of the bending energy. Hence such solutions are not stable. Studying the second variation allows us separate stable solutions from the unstable ones. For an example of this approach see~\cite{elastic}.

Another closely related application is motion planning. Given a control system one wants to transfer it from one state to another. Often there is an infinite number of ways of doing this (think of parking a car) and simple algorithms that allow to associate to two states a control function that steers the system between those two states are strongly desirable. One way to approach this problem is to look for minimum curves of some simple functional. Usually simplicity in this case means that the functional and the system has as many symmetries a possible which simplify its study. Then one can apply numerical algorithms or try to solve the problem by hand. Usually the last one is too difficult and a mixture of both is applied. Nevertheless with each new optimality condition we improve the speed of the convergence of the numerical solution, since we restrict to a much smaller set of candidates. This way we can obtain a better initial guess or even obtain a finite number of possibilities that can be checked by a computer. A classical example of this approach is the Reeds-Shepp car. Reeds and Shepp took the simplest model for a mobile robot and studied its time-optimal solutions~\cite{car}. They very able to reduce the number of candidate trajectories to a finite number of configurations. This number was later improved by Sussman and Tang using the technique of generalized envelopes~\cite{sussmann_car}. There is a deep connection between this technique and the theory we develop here. All optimality conditions from~\cite{sussmann_car} can be derived using our methods.

The second variation has also a purely geometric application. We can consider a variational problem on the space of curves of a differentiable manifold. For example, if the minimized functional is the length functional of a Riemannian metric, the theory of second variation allows to arrive naturally at the notion of sectional curvature. This approach was successfully applied in sub-Riemannian geometry, where curvature invariants that are directly related to the behavior of the geodesic flow were found~\cite{ABB}. It turns out that good models for constant curvature spaces are linear quadratic problems that are not sub-Riemannian manifolds, but which constitute a very natural class of optimal control problems. Using this class of systems in~\cite{davide_luca} the authors were able to prove comparison theorems, and one can hope that a similar theory can be developed for much more general constrained variational problems on manifolds.

Finally, many applications of PMP have produced various examples of bad behaving minimizers, which do not need to be smooth, which can accumulate an infinite number of discontinuities in finite time~\cite{zelikin} or even be chaotic~\cite{zelikin2}. Thus it may not be even possible to write down a Jacobi equation like in the situation above. Nevertheless surprisingly it was proven in~\cite{agr_lderiv_1,agr_lderiv} that there exists a Jacobi curve $L_t$ in the Lagrangian Grassmanian and a Maslov type theory that allows to codify all the information about the second variation. The basis of this general theory is the notion of $\cL$-derivatives. The goal of this paper is to explain this theory, to demonstrate some old and new ways to compute $\cL$-derivatives, how to construct the corresponding Jacobi curves and how to extract from them information about the Hessian of the functional along an extremal curve. We note that the notion of a conjugate point and Morse index theorems in optimal control were previously known for some specific cases of extremals like regular, singular or bang-bang (see, for example,~\cite{bonnard,goc,arona,dmitruk} and references there in). In this work we develop a general unified framework based on symplectic geometry, that can be applied to all the previously known cases and beyond.

This work is one of the many generalisations of the Morse index formula. There exists an infinite-dimensional version for elliptic boundary value problems introduced by Swanson~\cite{swanson} and further developed in~\cite{jones1,jones2,jones3}, a version for strongly-indefinite functionals (like Lorentzian geodesic problem) both in finite~\cite{portaluri,Waterstraat2,piccione} and infinite~\cite{Waterstraat1} dimensions. We should also mention various infinite-dimensional variants of Morse homology~\cite{smale,abb_maj} that greatly generalize the content and applicability of the Morse theorem we stated in the Introduction. Our main contribution is that we generalize the theory in the direction of much less regular variational problems as it is required by the applications that we have in mind. In particular, the extremals can have only $C^0$-regularity.

The article is organized as follows. In Section~\ref{sec:def} we discuss the definition of $\cL$-derivatives and Jacobi curves -- the main tool necessary for establishing the Morse type theorems at this level of generality. We introduce the optimal control setting, which is a natural language for any type of constrained variational problems on curves. In Subsection~\ref{sec:alg} we construct an algorithm for evaluating approximations of Jacobi curves with arbitrary good precision. This is a crucial step in the prove of Morse theorems. In Subsection~\ref{sec:sympl_def} we give a fairly detailed exposition of the Maslov index theory, including a special variant of the Maslov index of a triple of Lagrangian planes particularly suitable to our problem. In Subsection~\ref{sec:sympl_morse} we prove two Morse-type/spectral flow theorems: a simpler one that allows to compute the Morse index of the Hessian restricted to piecewise constant variations and a limiting Morse theorem, when the interval of constancy of variations goes to zero. At the end of the article the reader can find two appendices. In Appendix~\ref{app:chrono} we derives the necessary formulas for the Hessian. In Appendix~\ref{app:existence} we prove and state several auxiliary lemmas used in Section~\ref{sec:symp_and_morse} We hope that the tools that we develop in this article will find further applications in symplectic geometry and spectral flow theorems.

\section{$\cL$-derivatives for constrained variational problems}
\label{sec:l-deriv}
\subsection{General definitions}
\label{sec:def}

A $\cL$-derivative is a rule that for a given constrained variational problem assigns to an admissible space of variations a Lagrangian plane in some symplectic space. As we add variations we can compare the relative positions of the corresponding $\cL$-derivatives and deduce from that how the inertia indices and nullity of the Hessian change as we consider bigger and bigger spaces of variations. Using this theory one can recover the classical theory of Jacobi and much more. 

We consider the following setting. Let $J: \cU \to \R$ be a functional and $F: \cU \to M$ be a map with smooth finite-dimensional restrictions, where $\cU$ is a smooth Banach manifold and $M$ is a finite-dimensional manifold. Given a point $q\in M$, we are interested in finding $\tilde{\omega} \in F^{-1}(q)$ that minimize $J$ among all other points $\omega\in F^{-1}(q) $. In the case of optimal control problems $\cU$ is the space of admissible controls. The map $F$ is usually taken to be the end-point map, which we will introduce in the next subsection. 

The first step is to apply the Lagrange multiplier rule that says that if $\tilde{\omega}$ is a minimal point then there exists a covector $\lambda \in T_q^* M$ and a number $\nu\in\{0,1\}$, s.t. 
\begin{equation}
\label{eqlagrange}
\langle \lambda, dF[\tilde{\omega}](w) \rangle - \nu dJ[\tilde{\omega}](w) = 0,\qquad \forall w\in T_{\tilde{\omega}}\cU. 
\end{equation}

\begin{definition}
A pair $(\tilde{\omega},\lambda)$ that satisfies equation~\eqref{eqlagrange} is called a \textit{Lagrangian point} and $\tilde{\omega}$ is called a \textit{critical point} of $(F,J)$. If $\nu =0$ we say that the critical point is \textit{abnormal}, and if $\nu=1$ we call it \textit{normal}.
\end{definition}
There are of course many critical points that are not minimal. So in order to find the minimal ones we have to apply high order conditions for minimality. For example, we can study the \textit{Hessian} of a pair $(J,F)$ at a Lagrangian point $(\tilde{\omega},\lambda)$ defined as
\begin{equation}
\label{eqhessian}
\Hess (F,\nu J)[\tilde{\omega}, \lambda] := \left( \nu d^2 J[\tilde{\omega}] - \langle \lambda, d^2 F[\tilde{\omega}] \rangle \right)|_{\ker dF[\tilde{\omega}]}.
\end{equation}
The index and the nullity of the Hessian are directly related to optimality of the critical point $\tilde{\omega}$. Namely we have 
\begin{theorem}[\cite{as},Theorem 20.3]
Let $(F,J): \cU \to M \times \R$ be a pair of smooth maps and $\tilde{\omega}$ be a critical point with parameter $\nu$. Then if for any Lagrangian point $(\tilde{\omega},\lambda)$
$$
\ind^- \Hess (F,\nu J)[\tilde{\omega}, \lambda]  \geq \codim \im d(F,J)[\tilde \omega],
$$
the critical point $\tilde{\omega}$ is not optimal.
\end{theorem} 

In the normal case the Hessian of $(J,F)$ coincides with the second derivative of $J$ on the level set of $F^{-1}(q)$ at a Lagranian point $(\tilde{\omega},\lambda)$. Indeed, let $\omega(s)$ be a curve in $F^{-1}(q)$, s.t. $\omega(0) = \tilde{\omega}$. Then by differentiating twice $F(\omega(s)) = q$ at $s=0$ we find that
$$
dF[\tilde{\omega}](\dot{\omega})=0, 
$$
$$
d^2 F[\tilde{\omega}](\dot{\omega},\dot{\omega})+d F[\tilde{\omega}](\ddot{\omega})=0.
$$
Similarly we find that
\begin{align*}
\left.\frac{\p^2}{\p s^2}\right|_{s=0} J(\omega(s)) &= d^2 J[\tilde{\omega}](\dot{\omega},\dot{\omega})+d J[\tilde{\omega}](\ddot{\omega}) = d^2 J[\tilde{\omega}](\dot{\omega},\dot{\omega})+ \langle \lambda, dF[\tilde{\omega}](\ddot{\omega})\rangle =  \\
&= d^2J[\tilde{\omega}](\dot{\omega},\dot{\omega})-  \langle\lambda, dF[\tilde{\omega}](\dot{\omega},\dot{\omega})\rangle
\end{align*}
where in the second equality we have used that $(\tilde{\omega},\lambda)$ is a Lagrange point. From here we can see that this expression is equal exactly to $\Hess(F,\nu J)[\tilde{\omega},\lambda](\dot{\omega},\dot{\omega})$.

We are now ready to define $\cL$-derivatives. We linearise \eqref{eqlagrange} with respect to $\lambda$ and $\omega$, and obtain the following equation
$$
\langle \xi, dF[\tilde{\omega}](w) \rangle + \langle \lambda, d^2F[\tilde{\omega}](v,w) \rangle  - \nu d^2J[\tilde{\omega}](v,w)=0.
$$
Or if we define $Q(v,w) := \langle \lambda, d^2F[\tilde{\omega} ](v,w) \rangle - \nu d^2J[\tilde{\omega} ](v,w)$, we can rewrite this as
\begin{equation}
\label{eql_deriv_def}
\langle \xi, dF[\tilde{\omega}](w) \rangle + Q(v,w)=0.
\end{equation}
We note that $\Hess(F,\nu J)[\tilde{\omega},\lambda] = - Q|_{\ker dF[\tilde{\omega}]}$.

\begin{definition}
\label{def:l-deriv}
Let $(F,J)$ be maps with smooth finite-dimensional restrictions. A \textit{$\cL$-derivative} of $(F,J)$ at a Lagrangian point $(\tilde{\omega},\lambda)$ constructed over a finite-dimensional space of variations $V \subset T_{\tilde{\omega}} \cU$ that we denote as $\cL(F,\nu J)[\tilde{\omega},\lambda](V)$ is the set of vectors $(\xi,dF[\tilde{\omega}](v))\in T_{\lambda}(T^*M)$, s.t. $(\xi,v)\in (T_{\lambda}(T^*_{q}M), V)$ solve~\eqref{eql_deriv_def} for all $w\in V$.
\end{definition}

\begin{remark}
Let us give an informal explanation and motivation for this definition. Since $F: \cU \to M$ is a smooth map between two smooth manifolds, we can define the pull-back bundle $F^*(T^*M)$. All Lagrangian points satisfy equation~\eqref{eqlagrange}, and form a subset $\L \subset F^* (T^*M)$. Assume that this set is a manifold (this can indeed happen~\cite[Section 8.4.1]{ABB}). Then equation~\eqref{eqhessian} is just the characterization of the tangent space $T_{(\tilde{\omega},\lambda)}\L$ and $\cL(F,\nu J)[\tilde{\omega},\lambda](V)$ is the image of a subspace of $T_{(\tilde{\omega},\lambda)}\L$ under the differential of $F$. From here the second variation theory for smooth regular extremals can be established~\cite{ABB}. The goal of the rest of the subsection is to establish the basis for the theory even when $\L$ is not a smooth manifold.
\end{remark}

$\cL(F,\nu J)[\tilde{\omega},\lambda](V)$ is a Lagrangian plane if $V$ is finite-dimensional~\cite{agrachev_cime}. The reason why we do not take directly $T_{\tilde{\omega}} \cU$ instead of $V$ is that it is a linear equation defined on an infinite-dimensional space and it might be ill-posed. In this case $\cL(F,\nu J)[\tilde{\omega},\lambda](V)$ is just isotropic and can even be zero-dimensional~\cite{agrachev_moi}. But if we have chosen the right topology for our space of variations, we are going to get exactly $\dim M$ independent solutions. 

Lagrangian subspace $\cL (F,\nu J)[\tilde \omega,\lambda](V)$ contains information about the second variation restricted to the subspace $V$. To obtain a Lagrangian subspace that encodes the information about all the possible variations one has to use generalized sequences. For reader's convenience we recall some definitions. 

\begin{definition}
\textit{A directed set} $(I,\preccurlyeq)$ is a set $I$ with a preorder $\preccurlyeq$, s.t. for any two elements $\alpha,\beta \in I$ there exists an element $\gamma$, s.t. $\alpha \preccurlyeq \gamma$ and $\beta \preccurlyeq \gamma$.
\end{definition}
\begin{definition}
Given a directed set $(I,\preccurlyeq)$ a \emph{generalized sequence} or a net is a function from the set of indeces $I$ to a topological space $X$. A generalized sequence $\{x_\alpha\}_{\alpha \in I} \in X$ \emph{converges to a limit} $x\in X$, if for any open neighbourhood $O_x \ni x$ there exists an element $\beta \in I $, s.t. for all $\alpha \succcurlyeq \beta$ one has $x_\alpha \in O_x$.
\end{definition}

Finite dimensional subspaces of $T_{\tilde \omega} \cU$ form a directed set with a partial ordering given by the inclusion $V\subset W$. This motivates the following definition
\begin{definition}
A \emph{$\cL$-derivative} of $(F,J)$ at a Lagrange point $(\tilde \omega,\lambda)$ constructed over a subspace $V\subset T_u \cU$ is the generalized limit
$$
\cL (F,\nu J)[\tilde \omega,\lambda](V) = \lim_{W\nnearrow V}\cL (F,\nu J)[\tilde \omega,\lambda](W).
$$
taken over increasing finite-dimensional subspaces $W \subset V$.
\end{definition}
When $V$ is the whole space of available variations, we simply write $\cL(F,\nu J)[\tilde{\omega},\lambda]$ for the corresponding $\cL$-derivative. 

We have the following important theorem proved in~\cite{agr_lderiv}, that ensures the existence of this limit and gives a way to compute it.
\begin{theorem}
\label{thm:main}
Let $(\tilde{\omega},\lambda)$ be a Lagrangian point of $(F,J)$. 
\begin{enumerate}
\item If either the positive or the negative inertia index of $\Hess (F,\nu J)[\tilde \omega,\lambda]$ is finite, then $\cL(F,\nu J)[\tilde{\omega},\lambda]$ exists;
\item $\cL(F,\nu J)[\tilde{\omega},\lambda] = \cL(F,\nu J)[\tilde{\omega},\lambda](V)$ for any $V$ dense in $T_{\tilde{\omega}}\cU$.
\end{enumerate}
\end{theorem}

The second point is especially important, since it actually allows to compute the $\cL$-derivative. For example, if $\cU$ is modelled over a separable Hilbert space we can take a dense subspace spanned by vectors $e_1,e_2,...$ and define subspaces $V_i = \spn\{e_1,...,e_i\}$, $i\in \N$. Then we can compute the $\cL$-derivative as a limit
$$
\cL(F,\nu J)[\tilde{\omega},\lambda] = \lim_{i\to \infty}\cL(F,\nu J)[\tilde{\omega},\lambda](V_i)
$$
and there is no need for a generalized sequence in this case. We give the proof of this Theorem in Appendix~\ref{app:existence}, which extends the proof of the first author from~\cite{agr_lderiv_1} in greater detail than in the original article.

\begin{remark}
\label{rem:overdoing}
Another simple but important property is the following. Suppose that we have a smooth map $G: \cV \to \cU$ which is a submersion. And let $\tilde{v}$ be any preimage of a critical point $\tilde{\omega}$ under $G$. Then it is easy to check that $(\tilde{v},\lambda)$ is a Lagrange point of $(F\circ G, J\circ G)$ and that
$$
\cL(F,\nu J)[\tilde{\omega},\lambda] = \cL(F \circ G,\nu J\circ G)[\tilde{v},\lambda]. 
$$  
This property will be essential in the study of problems with constrains in the control.
\end{remark}

$\cL$-derivatives contain information about the inertia indices and nullity of the Hessian \eqref{eqhessian} restricted to some space of variations. The strength of this technique lies in the fact, that by comparing relative position of two $\cL$-derivatives constructed over two subspaces $V\subset W$, we can see how the inertia indices of the Hessian change as we add variations to our variations space. We have gathered some simple facts of this kind in Appendix~\ref{app:existence}.

\subsection{Optimal control problems with no constraints on the control}
\label{sec:optim}

Let us consider the following optimal control problem 
\begin{equation}
\label{eq:cont1}
\dot{q} = f(u,q), \qquad u \in U \subset \R^k, \qquad q \in M,
\end{equation}
\begin{equation}
\label{eq:func1}
J_T(u) = \int_0^{T} L(u,q) dt \to \min.
\end{equation}
We assume that we look for a minimum control in $L^\infty_k[0,T]$, that $f(u,q)$ and $L(u,q)$ are smooth in both variables. Moreover we assume that the final time $T$ is fixed and that $q(0)$ can be free and lie in some submanifold $N_0\subset M$, and $q(T) = q_T$ is fixed. We note that if $q(0)$ is fixed and $q(T)$ is not, then we transform our problem to a moving starting point by  making a change of time variable $t\mapsto T - t$. For now we simply assume that $U = \R^k$. We will see in the next subsection how the general case can be reduced to this one.

\begin{definition}
\textit{An admissible pair} $\omega(t)$ is a pair $(q(t),u(t))$ which consists of an absolutely continuous trajectory $q(t)$ that satisfies (\ref{eq:cont1}) for almost every $t$ and the corresponding control $u(t)$. We denote the set of admissible pairs by $\Omega$. 
\end{definition}

In~\cite{agrachev_cime} the following result was proven
\begin{proposition}
If $\dim M = n$, then the set of admissible pairs $\Omega$ has a structure of a smooth Banach manifold modelled over $\R^n \times L^\infty_k[0,T]$.
\end{proposition}

We don't give here a complete proof of this result, but we explain how one can construct an open neighbourhood of some admissible pair $\tilde{\omega}(t) = (\tilde{q}(t),\tilde{u}(t))$. Fix a moment of time $\tau \in[0,T]$ and consider an open neighbourhood $U$ of $\tilde{q}(\tau) = \pi(\tilde{\omega}(\tau))$ which is diffeomorphic to an open set in $\R^n$. Then if we fix $u(t)$, s.t. $||u(t) - \tilde{u}(t)||_\infty < \varepsilon $, through each point ${q}\in U$ passes a unique solution of (\ref{eq:cont1}) at a moment of time $\tau$, because of the well-posedness of the Cauchy problem. So one can see that locally a neighbourhood of $\tilde{\omega}(t)$ is a product of a small open neighbourhood of $\tilde{q}(\tau) $ in $M$ and an open neighbourhood of $\tilde{u}$ in $L^\infty_k[0,T]$. 

In order to construct the $\cL$-derivative we must specify a map that corresponds to our constraint.
\begin{definition}
\emph{The evaluation map} $F_t: \Omega \to M$ is a map, that is defined as
$$
F_t(\omega) = \pi(\omega(t)) = q(t).
$$
\end{definition}
This map is actually smooth because of the smooth dependence on parameters of the solutions to the Cauchy problem. Also from the construction of a neighbourhood of $\omega \in \Omega$ it is easy to see that $F_t$ is a submersion. Moreover the classical end-point map can be characterized as
$$
E_T = F_T |_{F_0^{-1}(q_0)}, \qquad q_0 \in M.
$$ 
The end-point map is the basic object in the study of optimal control problems. It takes a control $u(t)$ and maps it to the end of the corresponding trajectory that begins at $q_0 \in M$. Similarly if $q_0$ is not fixed, but lies in a manifold $N_0$, we define an analog of the classical end-point map as
$$
E_{N_0,T} = F_T |_{F_0^{-1}(N_0)}.
$$ 
We can then proceed to the construction of the $\cL$-derivative of $(E_{N_0,T},J_T)$, that will contain information about the Hessian of this problem. 

In the most well studied case of regular extremals $\cL(E_{T},\nu J_T)[\tilde{\omega},\lambda(T)]$ is the set of values of all Jacobi fields at a moment of time $T$ that have a zero projection on $TM$ at time $0$. If we assume that $q_0$ is free, then we need to modify accordingly the boundary condition at zero, but the meaning of the $\cL$-derivative is essentially the same. By constructing these $\cL$-derivatives for all times $t\in[0,T]$, and not just for the final time $T$, we reconstruct the whole set of Jacobi fields.

\begin{remark}
\label{rem:coordinates}
It is important to note that the definition of the end-point map, evaluation map, admissible curves etc. are all intrinsic, i.e. they do not depend on the choice of coordinates in $\cU$ and $M$. This means that the $\cL$-derivatives that we will construct are intrinsic as well and thus we can exploit the local structure of the space of admissible curves to simplify explicit computations. Previously we have discussed that the space of admissible curves is locally equivalent to $ L_k^\infty[0,T] \times \R^n$, which simply means that we look for the solutions of  \eqref{eq:cont1} with some control $u(t)\in L_k^\infty[0,T]$ passing through a point $q(\tau)\in M$, $\tau \in [0,T]$. But we can choose this $\tau$ as we want, the corresponding $\cL$-derivative will be the same for all $\tau$. This simplifies many things. For example, the inclusion of the space of admissible curves defined on an interval $[0,t] \subset [0,T]$ into the space of admissible curves defined on $[0,T]$ is simply given by taking the controls from $L_k^\infty[0,t] \subset L_k^\infty[0,T]$. Or by identifying a neighbourhood of $q(\tau)$ with $\R^n$ we can find coordinates s.t. $F_\tau(q(\tau),u(t)) = q(\tau)$. This implies that in this coordinate chart kernel of the differential of $F_\tau$ is exactly $L_k^\infty[0,T]$ and that the second derivative is zero. Finally we note that the space of variations $F_0^{-1}(N_0)$ for the very same reasons locally can be identified with $ L_k^\infty[0,T] \times \R^{\dim N_0}$.
\end{remark}

Let $\Omega^T_{N_0} = F^{-1}_0(N_0)$ be the space of all solutions of~\eqref{eq:cont1} starting at $N_0$ up to time $T$. If $\tilde{\omega}$ is a critical point of the pair $(E_{N_0,T},J_T)$, then we have
\begin{equation}
\label{eq:first_var}
\langle \lambda(T), d E_{N_0,T}[\tilde{\omega}](w)\rangle - \nu dJ_T[\tilde{\omega}](w)= 0, \qquad \forall w\in  \Omega_{N_0}^T.
\end{equation}
Here $(\lambda(T) , -\nu) \in T^*_{E_{N_0,T}(\tilde{\omega})} \times \R$ are the Lagrange multipliers, where $\nu$ is normalized to take values $0$ or $1$, $\tilde \omega \in F_0^{-1}(N_0)$. If we introduce the \textit{extended end-point map} $\hat{E}_{N_0,t} = (E_{N_0,t},J_t)$, we can rewrite this equation as
\begin{equation}
\label{eq:var_ext_end}
\langle \hat{\lambda}(t),d\hat{E}_{N_0,t}[\tilde{\omega}](w)\rangle = 0, 
\end{equation}
where $\hat{\lambda}(t) = (\lambda(t),-\nu)$. The extended end-point map $\hat{E}_{N_0,t}$ can be seen as the end point-map of the following control system
$$
\begin{array}{l}
\dot{q} = f(u,q)\\
\dot{y} = L(u,q)
\end{array} \qquad \iff \qquad \dot{\hat{q}} = \hat{f}(u,\hat{q}), 
$$
where $\hat{q}=(q,y)$. Let us denote the flow of this system with $u=\tilde{u}$ from time $t_0$ till time $t_1$ by $\hat P_{t_0}^{t_1}$. We also write $\hat P^t$ for $\hat P_0^t$. We use the non-hatted notation $P^t$ for the flow of the original control system (\ref{eq:cont1}) with the same control.

Since $d\hat{E}_{N_0,T}[\tilde{\omega}] |_{\Omega_{N_0}^t \cap L_k^\infty[0,t]} = (\hat P^T_t)_* d\hat E_{N_0,t} [\tilde{\omega}]$, by restricting~\eqref{eq:first_var} to $L_k^\infty[0,t]$ we find that
\begin{equation}
\label{eq:first_variation}
\langle \hat \lambda(t), d \hat E_{N_0,t}[\tilde{\omega}](w)\rangle = \langle \lambda(t), d E_{N_0,t}[\tilde{\omega}](w)\rangle- \nu dJ_t[\tilde{\omega}](w) = 0, \qquad \forall w\in  T_{\tilde{\omega}} \Omega_{N_0}^T,
\end{equation}
where $\hat \lambda(t) = (\hat P^T_t)^*\hat \lambda(T)$ and $\lambda(t)$ is the projection of $\hat{\lambda}(t)$ to $T^*_{\tilde{q}(t)}M$. Note that in the first inequality we have used the fact that the differential flow $\hat{P}^T_t$ leaves the subspace $(0,-\nu)$ invariant, since $\dot{y}$ does not depend on $y$.

If we define a Hamiltonian 
$$
h(u,\lambda) = \langle \lambda, f(u,q) \rangle - \nu L(u,q)
$$ 
then one can show~\cite{as}, that $\lambda(t)$ satisfies the Hamiltonian system 
$$
\dot{\lambda}(t) = \vec{h}(\tilde{u},\lambda(t)). 
$$
Moreover if we restrict the equation (\ref{eq:first_variation}) to $w\in  L_k^\infty[0,t]$ then we get a condition of the form
\begin{equation}
\label{eq:first_order}
\left.\frac{\p h(u,\lambda(t))}{\p u}\right|_{u=\tilde{u}} = 0,
\end{equation}
which is a weak form of the maximum principle. If we restrict the equation (\ref{eq:first_variation}) to $w\in  T_{\tilde{q}(0)} N_0$ we obtain instead the transversality conditions
\begin{equation}
\label{eq:trans}
\lambda(0) \perp T_{\tilde{q}(0)}N_0.
\end{equation}
For completness we give a derivation of these two conditions in Appendix~\ref{app:chrono}.

Let us consider an example in order to establish some relevant models. 

\begin{example}
\label{ex:lqr}

Let us consider a basic example from optimal control theory. Suppose that we are given matrices $A\in \Mat(n\times n)$, $B \in \Mat(n\times k)$ and a symmetric matrix $Q\in \Mat(n\times n)$. We wish to minimize 
$$
J_T[u] = \frac{1}{2}\int_0^T (|u|^2 - q^T Q q) dt
$$
for constraint system
$$
\dot{q}=Aq+Bu,
$$
where $q\in\R^n$, $u\in \R^k$. This is the famous linear quadratic problem (LQR). Let us write down the Hamiltonian vector field for normal extremals. Since $T^* \R^n \simeq \R^n\times \R^n$. We write $\lambda=(q,p)$. The Hamiltonian function is then given by
\begin{equation}
\label{eq:lqr_hamiltonian}
h(u,\lambda) = p^TAx+p^TBu - \frac{|u|^2-q^T Qq }{2}. 
\end{equation}
Using the canonical symplectic form $\sigma = \sum_{i=1}^n dp_i \wedge dq^i$, we then find that
\begin{equation}
\label{eq:lqr_vect}
\vec{h}(u,\lambda) = \begin{pmatrix}
-A^T & -Q\\
0 & A
\end{pmatrix}
\begin{pmatrix}
p\\
q
\end{pmatrix}
+ \begin{pmatrix}
0\\
Bu
\end{pmatrix}.
\end{equation}
If we now apply the weak form of the PMP~\eqref{eq:first_order}, then we get $\tilde{u}= B^Tp$. Thus
$$
\vec{h}(\tilde u,\lambda) = 
\begin{pmatrix}
-A^T & -Q\\
BB^T & A
\end{pmatrix}\begin{pmatrix}
p\\
q
\end{pmatrix}
$$
\end{example}

\begin{definition}
\label{def:jacobi_curve}
A \emph{Jacobi curve} of an optimal control problem (\ref{eq:cont1})-(\ref{eq:func1}) is the family of time parametrized Lagrangian subspace $\cL(E_{N_0,t},\nu J_t)[\tilde{\omega},\lambda(t)]$.
\end{definition}

It is important to note that Jacobi curves are feedback invariant. Indeed, $\cL$-derivatives remain unchanged under diffeomorphism of $\Omega_{N_0}$ as discussed in Remark~\ref{rem:overdoing} and feedback transformations just constitute a special case.

Definition~\ref{def:jacobi_curve} is quite natural in the light of the previous discussion, but it also has a small problem, namely $\cL$-derivatives from the definition are Lagrangian subspaces in different symplectic spaces $T_{\lambda(t)}(T^*M)$, and if we want to study their relative positions, we need a way to identify them. In order to fix this problem, let us consider the Hamiltonian flow of $\vec{h}_{\tilde{u}(t)}$, that we denote as $\Phi_t$. We are going to compute the pull-back $\Phi^*_t \cL(E_{N_0,t},\nu J_t)[\tilde{\omega},\lambda(t)](V)$. For brevity we denote this Lagrange subspace by $\cL_t(V)$. We will see that the information about the Hessian is encoded in some symplectic invariants of the Jacobi curve. Since $\Phi^*_t$ is a symplectomorphism, all the results about $\cL_t(V)$ will transfer automatically to the original invariant curve and we also gain the advantage that $\cL_t(V)$ stays in a fixed symplectic space $T_{\lambda(0)}(T^*M)$.

The next step is to write down explicitly the equation that defines $\cL_t(V)$. We have seen that the first order conditions are equivalent to the maximum principle with transversality conditions. Thus in order to obtain an explicit form for the equations \eqref{eq:first_variation} it is enough to linearize the Hamiltonian system, the maximum condition and the transversality conditions w.r.t. both phase variables and control variables. Since we are interested in the pull-back of the Jacobi curve under the pullback $\Phi^*_t$, we apply a time dependent change of variables $\mu = \Phi_t^{-1}(\lambda)$ on $T^*M$. Then the Hamiltonian system of PMP before the maximization is transformed to
$$
\dot \mu(t) = \vec{H}(t,u,\mu(t))= \left((\Phi_t)_*^{-1}(\vec{h}(u,\cdot)-\vec{h}(\tilde{u},\cdot))\right)(\mu(t)),
$$
where 
$$
H(t,u,\cdot) = (h(u,\cdot)-h(\tilde{u},\cdot))\circ \Phi_t.
$$
The maximum principle now says that any extremal control $u$ must satisfy the weak maximum condition \eqref{eq:first_variation}, which in the new coordinates has the same form as before
$$
\frac{\p H(t,u,\mu(t))}{\p u} = 0.
$$
Since $\Phi_0 = \id$, the transversality conditions have the same form as in \eqref{eq:trans}:
$$
\mu(0) \perp T_{\tilde{q}(0)} N_0.
$$

We note that under this change of variables, the Lagrange point $(\tilde{\omega},\lambda(t))$ is transformed to $(\tilde{q}(0), \tilde{u},\lambda(0))$ and from the formula for the new Hamiltonian we obtain $\vec{H}(t,\tilde{u},\lambda(0)) = 0$. Thus linearization at $(\tilde{q}(0),\lambda(0))$ will take a particularly simple form. 

Now we can finally write down explicit equations that define $\cL_t(V)$. Let us define a time-dependent vector field
$$
X(t) := \left.\frac{\p\vec{H}(t,u,\lambda(0))}{\p u}\right|_{u=\tilde{u}}
$$
and a quadratic form
$$
b(t)(v,w) := \left.\frac{\p^2 H(t,u,\lambda(0))}{\p u^2} \right|_{u = \tilde{u}}(v,w) = \left.\frac{\p^2 h(u,\lambda(t))}{\p u^2}\right|_{u = \tilde{u}}(v,w), \qquad \forall w\in \R^k,
$$

First we linearise the Hamiltonian system at $(\tilde{q}(0),\lambda(0))$. We get
\begin{equation}
\label{eq:ham_vec}
\dot{\eta}(t) = X(t) v(t)\quad \iff \quad\eta(t) = \eta_0 +\int_0^t X(\tau) v(\tau) d\tau.
\end{equation}
By identifying $T_{\lambda(0)}(T^*M)$ with $T_{\tilde{q}(0)}^*M \times T_{\tilde{q}(0)}M$, we obtain that the linearization of the transversality conditions gives
$$
\eta_0 \in T^\perp_{\tilde{q}(0)} N_0 \times T_{\tilde{q}(0)} N_0, 
$$
where $T^\perp_{\tilde{q}(0)} N_0 \subset T_{\tilde{q}(0)}^*M $ is just the annihilator of $T_{\tilde{q}(0)} N_0$.

Finally we linearize the maximum condition to obtain
$$
\left.\frac{\p^2 H(t,u,\lambda(0))}{\p u^2} \right|_{u = \tilde{u}}(v(t),w) + \left\langle \left. d_{\mu(t)}\frac{\p H}{\p u}\right|_{u=\tilde{u}} w, \eta(t) \right\rangle = 0, \qquad \forall w\in \R^k.
$$
Using the definitions we gave before, we can write
$$
\left\langle \left. d_{\mu(t)}\frac{\p H}{\p u}\right|_{u=\tilde{u}} w, \eta(t) \right\rangle = \sigma(\eta(t), X(t) w), \qquad \forall w\in \R^k.
$$

Collecting all the formulas proves the following result.
\begin{proposition}
\label{prop:def_eq}
An $\cL$-derivative $\cL_t(V)$ over a subspace $V\in L_k^\infty[0,t]$ consists of vectors of the form
\begin{equation}
\eta(t) = \eta_0 + \int_0^t X(\tau) v(\tau) d\tau. 
\end{equation}
where $\eta_0 \in T_{\tilde q_0}^\perp N_0 \times T_{\tilde q_0} N_0$ and $v\in V$ satisfy

\begin{equation}
\label{eq:master_or}
\int_0^t \left(\sigma \left(\eta_0 +  \int_0^\tau X(\theta) v(\theta) d\theta, X(\tau) w(\tau) \right) + b(\tau) (v(\tau), w(\tau)) \right)  d\tau = 0, \qquad \forall w(t) \in V.
\end{equation}
\end{proposition}

\begin{remark}
A more straightforward proof of this proposition can be done via explicit formulas for the first and second variation in Appendix~\ref{app:chrono}.
\end{remark}

\begin{example}
Let us revisit Example~\ref{ex:lqr}. First of all it is clear that $b(t) \equiv -\id_{k \times k}$ as can be seen directly from the Hamiltonian~\eqref{eq:lqr_hamiltonian}. Next we want to compute $X(t)$. We see that the Hamiltonian~\eqref{eq:lqr_hamiltonian} is a sum of four terms. Since in the definition of $X(t)$ we must differentiate both with respect to phase variables $(p,q)$ and control variables $u$, we obtain
\begin{equation}
\label{eq:form_X}
X(t) = \left.\frac{\p}{\p u}\right|_{u=\tilde{u}}(\Phi_t)^{-1}_*\overrightarrow{p^TB}(u-\tilde{u})(\mu(t))
\end{equation}
We need to calculate the flow $\Phi_t$, which is the flow of a linear non-autonomus equation:
\begin{align*}
\dot{q} &= Aq + B\tilde{u},\\
\dot{p} &= -A^T p - Q q,
\end{align*}
which gives us
$$
p(t) = e^{-tA^T}p_0 - e^{-tA^T} \int_0^t e^{\tau A^T} Q e^{\tau A} \left( q_0 + \int_0^\tau e^{-\theta A}B \tilde{u}(\theta) d\theta \right) d\tau.
$$
Now we just need to invert this expression with respect to $(p_0,q_0)$ and plug it in~\eqref{eq:form_X}. We obtain
$$
X(t)v = \begin{pmatrix}
e^{-tA}Bv\\
\int_0^t e^{\tau A^T} Q e^{\tau A}d\tau e^{-t A}Bv
\end{pmatrix}.
$$
From here formula~\eqref{eq:master_or} can be written explicitly.
\end{example}

As we have discussed before in the case when $U = \R^k$, the full $\cL$-derivative is defined as $\cL_t = \lim \cL_t(V)$.

We can compute $\cL_t$ using a dense subspace of $L_k^\infty[0,t]$. But one can also do the contrary and expand $L_k^\infty[0,t]$ to some weaker space. The $\cL$-derivative will not change if the first and the second differential are continuous in a weaker norm. One can note from formulas~\eqref{eq:first_der} and~\eqref{eq:second_var_g} in Appendix~\ref{app:chrono}, that the first and the second derivatives of $(E_{N_0,t},J_t)$ are actually continuous in $L^2_k[0,t]$. That is why from now on we use the space of square-integrable functions as our space of variations. This allows us to prove a simple, but important lemma.
\begin{lemma}
\label{lem:continuity}
The Jacobi curve $\cL_t$ is left continuous.
\end{lemma} 
\begin{proof}
Without any loss of generality we assume that all variations are two-sided. We compute $\cL_t$ over the space of piecewise constant functions with zero on the last interval. This space is dense in $L^2([0,t],\R^k)$ and therefore $\cL_t$ does not change. Fix a neighborhood $O_{\cL_t} \subset L(T_{\lambda(0)}(T^*M))$ of $\cL_t$. Then by definition of a generalized sequence there exists a finite-dimensional subspace $V$ of simple functions, s.t. for all $W\supset V$ one has $\cL_t(W\cap \Omega_{N_0}^t)\in O_{\cL_t}$. Let $\alpha = \{0=t_0 < t_1 < ... < t_N = t\}$ be the set of jump points of all variations $v(t)\in V$. By construction $v(t) = 0$ for $t\in [t_{N-1},t_N]$. We define $V_\beta \supset V$ to be the space of simple functions $v^\beta(t)$ with possible discontinuities in $\beta\supset\alpha$, s.t. $v^\beta(t) = 0$ for $t\in [t_{N-1},t_N]$. Then by definition $\cL_t(V_\beta\cap \Omega_{N_0}^t) \in O_{\cL_t}$ for any $\beta\supset\alpha$. By refining the partition $\beta$ on $[t_0,t_{N-1}]$ we obtain
$$
\lim \cL_t(V_\beta\cap \Omega_{N_0}^t) = \cL_t(\Omega_{N_0}^{t_{N-1}}) = \cL_{t_{N-1}} \in O_{\cL_t}.
$$
Since $t_{N-1}$ can be arbitrary close to $t$, the result follows.
\end{proof}

\subsection{Optimal control problems with constraints}

Let us now consider the case when the set $U\subset \R^k$ is a disjoint union 
$$
U = \bigcup_{i=1}^N U_i
$$
of closed embedded submanifolds $U_i \subset \R^k$ without boundary. A typical example in control theory is a curve-linear polytope in $\R^k$ defined by a number of inequalities
$$
p_i(u) \leq 0
$$ 
that satisfy 
$$
p_i(u) = 0 \quad \Leftrightarrow \quad d_u p_i = 0.
$$
Then $U$ is union of the interior of the polytope and faces of different dimensions. 

In order to compute the $\cL$-derivatives we need to take derivatives of the functional and the end-point map, which implies that we use two sided variations. But when we are on the boundary of the manifold of admissible pairs, we can not variate in all directions. Definition~\ref{def:jacobi_curve} still makes sense if we just compute $\cL(E_{N_0,t},\nu J_t)[\tilde{\omega},\lambda(t)](V)$, where $V$ is the set of admissible two-sided variations. However, it does not fix the problem entirely that can be seen very well from the example when $U$ is a polytope and the extremal curve is takes values in its vertices. Such trajectories can be optimal and are called \textit{bang-bang} trajectories. There are no two-sided variations at all. To solve this issue, we introduce additional variations in our problem so that the set of two sided variations is never empty and define Jacobi curves as $\cL$-derivatives constructed over it. This will be enough to enclose all the known results and to generalize them.

Variations that we need are called \textit{time variations} or internal variations. Basically we introduce a new time variable $\tau$, given by
$$
t(\tau) = \int_0^\tau (1 + u_0(s)) ds.
$$
We assume that $u_0(s) > -1$ and we take the time variable as a new state variable
$$
\dot{t}= 1+ u_0(\tau)
$$
since the final time is fixed. 

Under these assumptions function $t(\tau)$ is strictly increasing and therefore invertible. Then our control system is transformed to
\begin{align*}
\frac{d q}{d \tau} &= (1  + u_0(\tau)) f(q,u(t(\tau))), \\
\dot{t} &= 1+ u_0(\tau)
\end{align*}
and the functional to
$$
\int_0^{\tau(T)} (1 + u_0) L(q,u) ds \to \min.
$$

If $\tilde u$ was an optimal control for (\ref{eq:cont1})-(\ref{eq:func1}), then $(\tilde u,0)$ will be optimal for the new problem. Thus after we have included time variations, we just construct the $\cL$-derivative at $((\tilde{u},0),\lambda(t))$ over the set of all available two sided-variations, which is now non-empty.

\begin{remark}
\label{rmk:time}
Time variations were previously used to derive necessary and sufficient condition of bang-bang arcs~\cite{agrachev_bang,agrachev_bang2} and with small modification of their definition, one can even prove a version of the maximum principle~\cite{dmitruk2}. These variations actually do not give contribution to the index of the second variation if the considered control $\tilde{u}(s)$ has at least $C^2$-regularity (see the discussion in~\cite{agrachev_moi}). But if the control has less regularity like in the bang-bang case, time variations allow to find necessary optimality conditions even when there are not enough two-sided variations.
\end{remark}

From now on we assume that the time variations are included in the formulation of the problem and that consequently the space of two-sided variations is non-empty. 

Now we construct the $\cL$-derivatives. Note that since $U$ is embedded in $\R^k$, the operators $X(t): \R^k \to T_{\lambda(0)}(T^*M)$ and quadratic forms $b(t): \R^k \times \R^k \to \R$ are well defined. Now we must enlarge the space of variations by including time variations. After doing this, we remain in the same class of problems as before. So we keep the notations $X_t$ and $b_t$ for the Hamiltonian vector field and the corresponding Hessian for the new system defined above.

Since we are interested only in two-sided variations, we must restrict the operator $\tilde X(t)$ to the tangent spaces of $U_i$. Since by assumptions each $U_i$ is embedded in $\R^{k}$, let us choose any metric in the ambient space and for each point $u\in U_i$ we define an orthogonal projection $\pi^i_u: \R^k \to T_{u}U_i$ that depends on a point. Then we use this to define a projection of a given variation to the subspaces of two-sided variations as
$$
\pi_{\tau}v(\tau) = \sum_{i=1}^n \chi_{U_i}(\tilde u(\tau))\pi^i_{\tilde u(\tau)} v(\tau),
$$   
where $\chi_{U_i}$ is the indicator function of $U_i$. 

We can see that finite-dimensional approximations to the $\cL$-derivative will depend on the choice of the metric in $\R^k$, but the limit $\cL$-derivative itself will not, since we approximate the same space of variations in two different ways. In order to reduce the problem to $U=\R^k$, we simply replace $X(t)$ and $b(t)$ by $X(t) \pi_t$ and $b(t)(\pi_t\cdot,\pi_t\cdot)$. By doing this, we essentially exploit Remark~\ref{rem:overdoing}. 

This way we have reduced our problem to a problem without the constraints on the control and from now on we assume that the variations $v(t)$ can take any value in $\R^k$. To simplify the notations we continue to write $X(t)$ and $b(t)$ instead of $X(t)\pi_t $ and $b(t)\pi_t $.

\begin{remark}
We note that Jacobi curves will depend on the decomposition of $U$ into a disjoint union of manifolds $U_i$. For examples, if $U$ is a union of two intersecting lines, we can assume that the intersection point belongs either to one line or the other. But for the most common choice of $U$ as a polytope, the most natural decomposition would be just to take as $U_i$ different faces of various dimensions.
\end{remark}

Let us consider some examples of variational problems which show non-smooth behaviour in order to motivate further the necessity of the developed theory.

\begin{example}
Assume that we want to minimize
$$
\int_0^\infty q(t)^2 dt\to\min
$$
subject to constraints
$$
\ddot{q}(t) = u(t), \qquad |u|\leq 1, \qquad q\in \R, \qquad u \in \R,
$$
and boundary conditions
\begin{equation}
q(0) = q_0, \qquad \dot{q}(0) = \dot{q}_0. 
\end{equation}
It turns out that the minimal control in this case is an infinite number of switches between constant controls $u = \pm 1$. The intervals of constancy are shorter and shorter in time and the trajectory arrives at the origin with an infinite number of switches in finite time. This is so-called \textit{Fuller phenomena}, which has many applications in engineering~\cite{zelikin}.

\end{example}

\begin{example}
The previous example might create an impression that such a singular behavior arises from the fact that the boundary of the set $U$ has several connected components or due to some other singularity. However discontinuities can appear even if all of the data is smooth. For example, in~\cite{caillau,carolina} the authors considered the following optimal control problem: minimize time needed to connect two points on a manifold $M$ via a control system
$$
\dot{q}=X_0(q) + \sum_{i=1}^k u_i X_i(q),
$$
with $U$ being the unit ball. The authors showed that optimal controls can indeed have jumps. 
\end{example}

\subsection{An algorithm for computing the $\cL$-derivative}
\label{sec:alg}

Equations~\eqref{eq:master_or} can be used to construct approximations of $\cL$-derivatives, which is morally the Galerkin method. In this subsection we will see that if we choose piecewise constant functions as our basis elements, then the finite dimensional approximation of~\eqref{eq:master_or} takes a nice block triangular form, which can be solved explicitly, yielding an effective algorithm for constructing approximations of the Jacobi curve.

The $\cL$-derivative for optimal control problems enjoys several useful properties. First of all we have seen that the $\cL$-derivative $\cL_T$ exists if $\ind^\pm \Hess ( E_{N_0,T},\nu J_T)[\tilde{\omega},\lambda(T)] < +\infty$. But since $E_{N_0,T} |_{\Omega_{N_0}^t}= E_{N_0,t}$ for $t\leq T$ and $L_k^2 [0,t]\subset L_k^2[0,T]$ is an isometrical embedding, we have that the existence at a moment of time $T$ implies the existence for all $t \leq T$. 

Next we prove one more useful property that greatly simplifies the computation of $\cL_t$. 

\begin{lemma}[Additivity]
\label{lemm:add}
Take $0 < t_1  < t_2$ and suppose that the index of the Hessian of the extremal curve on $[0,t_2]$ is finite. We denote by $V_2$ some finite dimensional subspace of $L^2_k[t_1,t_2]$ and we consider the following equation
\begin{equation}
\label{eq:master_short}
\int_{t_1}^{t_2} \left[ \sigma \left( \lambda + \int_{t_1}^{\tau} X(\theta) v_2(\theta) d\theta, X(\tau)  w(\tau) \right) + b(\tau)\left( v_2(\tau),  w(\tau)\right) \right] d\tau = 0, \qquad \forall w(\tau) \in V_2,
\end{equation}
where $v_2(\tau) \in V_2$, and $\lambda \in \cL_{t_1}$. 

Then $\cL_{t_2}$ is a generalized limit of Lagrangian subspaces
$$
\left\{
\lambda + \int_{t_1}^{t_2} X(\tau) v(\tau) d\tau : \lambda \in \cL_{t_1} \;,\; v(t)\in V_2 \textit{ satisfies (\ref{eq:master_short}) for any } w(t)\in V_2 \right\}.
$$

\end{lemma}

\begin{proof}
By the existence theorem and the remark above we know that $\cL_{t_2}$ and $\cL_{t_1}$ exist and $\cL_{t_2}$ can be computed over any dense subspace of the variation space. So we compute it over $V_1 \oplus V_2 = V \subset \Omega_{N_0}^{t_2}$, where $V_1$ is a span of a countable dense subset in $\Omega_{N_0}^{t_1}$ that includes variations of the initial point. Denote by $\pi_i$ the projection onto $V_i$.

Now fix a neighbourhood $O_{\cL_2}$ in the Lagrangian Grassmanian and consider a finite-di\-men\-sional subspace $W\subset V$, s.t. for any finite dimensional $U\supset W$ we have $\cL_t(U)\in O_{\cL_2}$. Then we can construct a countable sequence of nested subspaces 
$$U_1 \subset U_2 \subset ...$$
by adding vectors from the basis of $V_1$. As a result we get a sequence $\cL_t(U_i)$ which converges to $\cL_t(V_1\oplus \pi_2 (W))$, since the index over the Hessian on this subspace must be finite as well. Note that in this case by construction $\cL_t(V_1\oplus \pi_2 (W))\subset O_{\cL_2}$. By taking finer and finer $O_{\cL_2}$ we realize $\cL_{t_2}$ as a limit of vectors from $\cL_t(V_1\oplus \pi_2 (W))$.

It remains to show that (\ref{eq:master_short}) holds. And indeed, any element of $\cL_t(U_i)$ is of the form
$$
\eta + \int_{0}^{t_1} X(\tau) v_1^i(\tau) d\tau  +\int_{t_1}^{t_2} X(\tau) v_2(\tau) d\tau, \qquad v_1^i \in V_1 \cap U_i \qquad v_2 \in \pi_2 (W)  
$$
s.t.
\begin{align*}
&\int_0^{t_1} \left[\sigma(\eta + \int_0^{\tau} X(\theta) v_1^i(\theta) d\theta,X(t) w_1(\tau)) + b(\tau)(v_1^i(\tau),w_1(\tau))\right]d\tau = 0\\
&\int_{t_1}^{t_2} \left[\sigma(\eta + \int_0^{t} X(\theta) v_1^i(\theta) d\theta + \int_{t_1}^{\tau} X(\theta) v_2(\theta) d\theta,X(\tau) w_2(\tau)) + b(\tau)(v_2(\tau),w_2(\tau))\right]d\tau = 0 
\end{align*}
for any $w_1 \in V_1\cap U_i$, $w_2 \in \pi_2 (W)$. Therefore as we take the limit, the vectors
$$
\eta + \int_0^{t_1} X(\tau) v^i_1(\tau) d\tau
$$
will converge to vectors from $\cL_t(V_1)$.

\end{proof}

These properties are enough to have an algorithm for computing $\cL_t$ at each moment of time $t$ with arbitrary good precision. Since we can replace $L^2_k[0,t]$ with any dense subset, we compute $\cL_t$ over the space of piecewise constant functions. To construct an approximation of $\cL_t$ we just have to take some partition $D = \{0 <t_1<t_2<...<t_N = t\}$ of $[0,t]$ and construct $\cL_t(V_D)$, where $V_D\subset \Omega_{N_0}^t$ is the space of variations of the initial point and piecewise constant variations of the control with jumps at $D$. Then we can use the additivity lemma to iteratively construct an approximation to $\cL_t(V_D)$, given by $T_{\tilde{q}(0)}^\perp N_0 \times T_{\tilde{q}(0)} N_0= \cL_t(V^{\{0\}})$, $\cL_t(V^{\{0,t_1\}})$, $\cL_t(V^{\{0,t_1,t_2\}})$ and so on. So at the end we just need to understand how $\cL_t$ changes when we add constant variations $\R^k \chi_{[t,t+\varepsilon]}$. In this case at each step we need to solve an over-determined finite-dimensional linear system. A convenient machinery for such type of equations is the notion of pseudo-inverses. We recall their basic definition.

\begin{definition}
Let $A: \R^m \to \R^n$ be a linear map between two Euclidean spaces and $A^*$ be its adjoint. Then \textit{the Moore-Penrose pseudoinverse} $A^+$ can be defined as
$$
A^+ = \lim_{\varepsilon\to 0} (\varepsilon \id + A^*A)^{-1}A^*.
$$
\end{definition}

The Moore-Penrose inverse has many interesting properties. The most useful one for us will be the following one.
\begin{proposition}
If the linear solution $Ax = b$ admits at least one solution, then $y=A^+b$ is the minimal norm solution of this equation.
\end{proposition}

Now we are ready to describe our algorithm.
\begin{theorem}
\label{thm:algorithm}
Suppose that we know $\cL_t(V)$, where $V$ is some space of variations defined on $[0,t]$. We identify $\cL_t(V)$ with $\R^n$ and the space of control parameters with $\R^k$, and put an arbitrary Euclidean metric on both of them. Let $E$ be the space of all $v\in \R^k$ for which
$$
\sigma\left( \eta, \frac{1}{\varepsilon}\int_t^{t+\varepsilon} X(\tau) d\tau \cdot v\right) = 0, \qquad \forall \eta \in \cL_t(V)
$$
and let $L\subset \cL_t(V)$ consisting of all $\eta\in \cL_t(V)$, s.t.
$$
\sigma\left( \eta, \frac{1}{\varepsilon}\int_t^{t+\varepsilon} X(\tau) d\tau \cdot w \right) = 0, \qquad \forall w \in \R^k.
$$
We define two bilinear maps $A_R: \cL_t(V) \times E^{\perp} \to \R $, $ Q_R:  E^{\perp} \times E^{\perp}  \to \R$:
\begin{align*}
A_R: (\eta, w) &\mapsto \sigma  \left( \eta, \frac{1}{\varepsilon}\int_t^{t+\varepsilon} X(\tau) d\tau \cdot w \right), \\
 Q_R: (v,w) &\mapsto \frac{1}{\varepsilon}\int_{t}^{t+\varepsilon}\sigma \left( \int_t^\tau X(\theta) d\theta \cdot v, X(\tau) w\right) + b(\tau)(v,w) d\tau,
\end{align*}
and we use the same symbols for the corresponding matrices. 

Then the new $\cL$-derivative $\cL_{t+\varepsilon}(V \oplus \R^k \chi_{[t,t+\varepsilon]})$ is a span of vectors from the subspace  $L$ and vectors 
$$
\eta_i + \frac{1}{\varepsilon} \int_t^{t+\varepsilon} X(\tau) d\tau \cdot v_i, 
$$
where $v_i$ is an arbitrary basis of $E^\perp$ and $\eta_i$ are defined as
\begin{equation}
\label{eq:solution}
\eta_i = - A^+_R Q_R v_i.
\end{equation}
\end{theorem}

\begin{proof}
From the additivity lemma it follows that it is sufficient to construct $n$ independent solutions of the equation
\begin{equation}
\label{eq:master}
\int_{t}^{t+\varepsilon}\sigma \left(\eta + \frac{1}{\varepsilon}\int_t^\tau X(\theta) d\theta \cdot v, X(\tau) w\right) + \frac{b(\tau)(v,w)}{\varepsilon} d\tau = 0, \qquad \forall w \in \R^k,
\end{equation}
where $\eta \in \cL_t(V), v\in \R^k$, modulo the relation
$$
\eta+ \int_t^{t+\varepsilon}X(\tau)v(\tau)d\tau = 0.
$$ 
Using the fact that $\eta \in \cL_t(V)$, we can rewrite this condition as
$$
\int_{t}^{t+\varepsilon}\sigma \left(\eta, X(\tau) v\right) d\tau = 0, \qquad \forall \eta \in \cL_t(V),
$$
i.e. $v \in E$. It means that we need to solve~\eqref{eq:master} with $v\in E^\perp$. But this way we have $n+\dim E^\perp$ variables and only $k$ equations. So if we want to have at least $n$ independent solution, there should be $k-\dim E^\perp$ dependencies, i.e. there exist $k-\dim E^\perp$ independent vectors $w_i \in \R^k$ such that
$$
\int_{t}^{t+\varepsilon}\sigma \left(\eta + \frac{1}{\varepsilon}\int_t^\tau X(\theta) d\theta \cdot v, X(\tau) w_i\right) + \frac{b(\tau)(v,w_i)}{\varepsilon} d\tau = 0, \qquad \forall v \in E^\perp, \forall \eta \in \cL_t(V).
$$
In particular
$$
\int_{t}^{t+\varepsilon}\sigma \left(\eta, X(\tau) w_i\right) d\tau = 0, \qquad \forall \eta \in \cL_t(V),
$$
and hence $w_i$ must form a basis of $E$. This means that we can reduce~\eqref{eq:master_or} to the same equations, but with $v,w\in E^\perp$.

Moreover by definition we have
$$
L = \cL_t(V)\cap \cL_{t+\varepsilon}(V \oplus \R^k \chi_{[t,t+\varepsilon]}) = \left\{\eta \in \cL_t(V): \sigma\left( \eta, \int_t^{t+\varepsilon} X(\tau) d\tau \cdot w \right)= 0, \, \forall w \in \R^k \right\}.
$$
So we only need to find $n-\dim L$ vectors $(\eta,v) \in L^\perp\times E^\perp$ which solve the reduced equation
$$
A_R \eta + Q_R v = 0. 
$$
By construction $\ker A_R = \{0\}$. This means that for each $v\in R^\perp$ such that $Qv\in \ran A$ there exists a unique $\eta \in L^\perp$ which solves this equation and the solution is given by~\eqref{eq:solution}. Since there are $n-\dim L$ independent solutions of this equation, we must have $\dim E^\perp \geq n - \dim L = \dim L^\perp$. But both those dimensions are actually equal. Indeed let us denote by $A: \cL_t(V) \times \R^k \to \R $ the bilinear form
$$
(\eta,w)\mapsto \sigma\left( \eta, \int_t^{t+\varepsilon} X(\tau) d\tau \cdot w \right),
$$ 
as well as the corresponding matrix. It is clear that $L = \ker A$ and $E=\ker A^*$. Since the column rank and row rank are equal we have $\dim L^\perp = \dim E^\perp$.

\end{proof}

We stress once again that the $\cL$-derivative itself is invariant and does not depend on the choices we make. The proven theorem is going to play an essential role in the Morse-type theorems that we are going to state and prove in Section~\ref{sec:symp_and_morse}.


\section{Symplectic geometry and Morse type theorems}
\label{sec:symp_and_morse}

Under Morse type theorems we understand results that relate the behaviour of the Jacobi curve or its approximation with the index or the kernel of the corresponding Hessian. Before stating these results we need some definitions and theorems from symplectic geometry. We should note that the sign conventions in this work differ from the sign conventions in some of the previous articles like~\cite{agrachev_quadratic,agr_gamk_symp} and some of the standard textbooks~\cite{sal_mcd,gosson}. The reason for this is that we would like to handle the normal and abnormal cases in a unified manner, and that in the smooth cases we want the Morse index of the Hessian to be equal to the Maslov index of the related Jacobi curve. This fact will determine a fixed co-orientation on the Maslov train $\cM_\Pi$, which will be defined in this section. This way the classical Jacobi equation for regular extremals produces a monotone increasing curve in the Lagrangian Grassmanian in the sense of Definition~\ref{def:monotone}.

\subsection{Linear symplectic geometry}
\label{sec:sympl_def}

Let $(\Sigma,\sigma)$ be a symplectic space and denote by $L(\Sigma)$ the corresponding Lagrangian Grassmanian. We will denote by $\Pi^\pitchfork$ the set of all Lagrangian planes transversal to a given Lagrangian plane $\Pi \in L(\Sigma)$.

Fix $\Pi\in L(\Sigma)$ and choose a plane $\Delta \in \Pi^\pitchfork$ transversal to it. Then $\Sigma = \Pi \oplus \Delta$ and any $L\in \Pi^\pitchfork$ can be identified with a graph $(x,Sx)$ of a map $S:\Delta \to \Pi$. One can easily check that $S$ defines a Lagrangian subspace if, and only if, $S$ is symmetric. This construction gives a local chart, which allows to identify locally $L(\Sigma)$ with the space of symmetric matrices. 

A more intrinsic but essentially the same way to identify a neighbourhood of $\Delta \in \Pi^\pitchfork$ with symmetric quadratic forms goes as follows. Let $P_\Lambda$ be the projection operator from $\Delta$ to $\Lambda \in \Pi^\pitchfork$ parallel to $\Pi$. Then we identify $\Lambda$ with
$$
Q_\Lambda(\lambda) = \sigma(P_\Lambda\lambda,\lambda), \qquad \lambda\in\Delta. 
$$
A simple equality
$$
\sigma(P_\Lambda \lambda, \mu) + \sigma(\lambda,P_\Lambda  \mu) = \sigma(\lambda,\mu),\qquad \forall \lambda,\mu\in\Sigma.
$$
implies that $Q$ is symmetric. One can prove, that for any $\Lambda\in \Pi^\pitchfork$ there exists a unique projection operator $P_\Lambda: \Sigma \to \Lambda$, s.t. the equality above holds and $P_\Lambda \Pi = \{0\}$. For more details see~\cite{sternberg}.

Similarly one has an identification of the tangent space $T_{\Lambda} L(\Sigma)$ with the space $\Sym(\Lambda)$ of all symmetric quadratic forms on $\Lambda$. Indeed, given $\Lambda$ take any curve $\Lambda(\varepsilon) \in L(\Sigma)$, s.t. $\Lambda(0) =\Lambda$, fix $\lambda(0) \in \Lambda(0)$ and take any curve $\lambda(\varepsilon) \in \Lambda(\varepsilon)$. Then we identify $\dot{\Lambda}_0$ with the quadratic form $\sigma(\lambda(0),\dot{\lambda}(0))$. The definition does not depend on the choice of the curve $\lambda(\varepsilon)$. Indeed, the definition is local so we can choose any coordinate chart centred at $\Lambda(0)$ and represent it using a curve of symmetric matrices $S:[-1,1]\to \sym(n)$. Then we can write $\lambda(s) = (p(s),S(s)p(s))$, for some curve $p:[-1,1]\to\R^n$. From here we compute
\begin{align*}
\sigma(\lambda,\dot\lambda)|_{s=0}&=\left.\sigma\left((p,Sp),(\dot p,\dot Sp+S\dot p)\right)\right|_{s=0} = \\
&= \left.\left(p^T(\dot Sp+S\dot p) -  \dot p^T S p \right) \right|_{s=0} = p^T(0)\dot S(0)p(0).
\end{align*}

\begin{definition}
\label{def:monotone}
We say that a $C^1$-curve $\Lambda(t) \in L(\Sigma)$ is \emph{monotone increasing} if the corresponding matrix $\dot{\Lambda}(t) > 0$ as a quadratic form on $\Lambda(t)$ for every $t$. We say that it is \emph{strictly monotone} if this inequality is strict.
\end{definition}
We can define similarly monotone decreasing curves, but it turns out that $C^1$-smooth Jacobi curves of minimum problems are always monotone increasing.

Now we define the Maslov index of a continuous curve. The \textit{Maslov train} $\cM_\Pi$ is the set of all Lagrangian planes non-transversal to $\Pi \in L(\Sigma)$, i.e. $\cM_\Pi = L(\Sigma) \setminus \Pi^\pitchfork$. This is a stratified manifold where each strata $\cM^k_\Pi$ is the set of all Lagrangian planes $\Lambda\in L(\Sigma)$ s.t. $\dim \left( \Lambda \cap \Pi \right) = k$. The dimension of each strata is 
$$
\dim \cM_\Pi^k = \frac{n(n+1)}{2} - \frac{k(k+1)}{2}.
$$ 
We can see that the highest dimensional strata $\cM_\Pi^1$ has codimension one in $L(\Sigma)$. To define an intersection index we need to define a co-orientation on $\cM_\Pi^1$. Suppose that $\Lambda(\varepsilon) \in L(\Sigma)$ intersects $\cM_\Pi^1$ transversally at $\varepsilon = 0$, i.e. there exists unique up to a scalar factor $\lambda \in \Lambda(0) \cap \Pi$. We define a positive co-orientation when 
$$
\dot{\Lambda}(0)(\lambda)>0.
$$
Similarly one defines a negative co-orientation. 

\begin{definition}
\emph{The Maslov index} $\Mi_\Pi(\Lambda(t))$ of a curve $\Lambda(t)$ is the intersection number of $\Lambda(t)$ with $\cM_\Pi^1$. For a curve in general position this is just a number of intersections $\Lambda(t)\cap \Pi$ counted with signs.
\end{definition}
Since $\cM_\Pi^2$ has codimension three, the Maslov index is well defined and it is a homotopy invariant. The importance of the Maslov index comes from the following fact.
\begin{proposition}
\label{prop:maslov_closed}
The Maslov index $\Mi_\Pi$ as a function on the loops in the Lagrangian Grassmanian $L(\Sigma)$ induces an isomorphism $\pi_1(L(\Sigma))\to\Z$. Moreover this isomorphism does not depend on the choice of $\Pi$, i.e.
$$
\Mi_{\Pi}(\Lambda(t)) = \Mi_{\Delta}(\Lambda(t)), \qquad \forall \Pi,\Delta \in L(\Sigma).
$$
\end{proposition}
For the proof and more properties see~\cite{agrachev_moi}. Often we will just write $\Mi(\Lambda(t))$ for closed curves to emphasize, that the result does not depend on the choice of the Maslov train.

The given definition is very useful in many theoretical studies, but not very convenient for computations, since one needs to put the curve in a general position and verify that the boundary points are not in $\cM_\Pi$. To overcome this, one usually uses other symplectic invariants of Lagrangian planes and curves. We will need the Kashiwara index $\Ki(\Lambda_1,\Lambda_2,\Lambda_3)$ of a triple of Lagrangian planes $\Lambda_i\in L(\Sigma)$ and the Leray index $\Li(\tilde\Lambda_1,\tilde\Lambda_2)$ of two points $\tilde\Lambda_i$ in the universal cover $\widetilde{L(\Sigma)}$ to state and prove the main Morse index theorem, but for many intermediate steps it is more useful to use an index introduced in~\cite{agrachev_quadratic}.

To define it we take three Lagrangian planes $\Lambda_1,\Lambda_2, \Pi \in L(\Sigma)$ and define a quadratic form $q$ on $((\Lambda_1 + \Lambda_2) \cap \Pi)/(\Lambda_1\cap \Pi \cap \Lambda_2)$ as
\begin{equation}
\label{eq:form_index}
q(\lambda) = \sigma(\lambda_1,\lambda_2), \qquad \lambda = \lambda_1 + \lambda_2, \qquad \lambda_i \in \Lambda_i.
\end{equation}
\begin{definition}
The \emph{positive Maslov index} of a triple $(\Lambda_1,\Pi,\Lambda_2)$ is a half-integer number
\begin{align*}
\ind_\Pi(\Lambda_1,\Lambda_2) &= \ind^+ q + \frac{1}{2}\dim \ker q = \\
&= \ind^+ q + \frac{1}{2}\left( \dim\left( \Lambda_1\cap \Pi \right) + \dim\left( \Lambda_2\cap \Pi \right) \right) - \dim(\Lambda_1\cap\Lambda_2 \cap \Pi).
\end{align*}
\end{definition}

The positive Maslov index has many important properties. We list just a few and refer to~\cite[Pages 2684-2688]{agrachev_quadratic} for some others and the proofs. We note again that in~\cite{agrachev_quadratic} a different sign convention was used and therefore the negative Maslov index played the central role.
\begin{lemma}
\label{lemm:maslov} The positive Maslov index has the following properties for all $\Lambda_i,\Pi \in L(\Sigma)$
\begin{enumerate}
\item Explicit finite bounds
$$
0 \leq \ind_{\Pi}(\Lambda_1,\Lambda_2) \leq \frac{\dim \Sigma}{2};
$$
\item If $\Gamma \subset \Lambda_1\cap \Lambda_2$, we denote $\Pi^\Gamma = (\Pi\cap \Gamma^\angle) + \Gamma$, where is the skew-orthogonal complement with respect to the symplectic form $\sigma$ defined as
$$
\Gamma^\angle = \left\{\lambda \in \Sigma \,:\, \sigma(\lambda,\mu) = 0, \forall \mu \in \Gamma  \right\}.
$$
Then
$$
\ind_\Pi(\Lambda_1,\Lambda_2) = \ind_{\Pi^\Gamma}(\Lambda_1,\Lambda_2);
$$
\item Triangle inequality
$$
\ind_{\Pi}(\Lambda_1,\Lambda_3) \leq \ind_{\Pi}(\Lambda_1,\Lambda_2) + \ind_{\Pi}(\Lambda_2,\Lambda_3);
$$
\item A formula
$$
\ind_{\Pi}(\Lambda_1, \Pi) = \ind_{\Pi}(\Pi,\Lambda_1) = \frac{1}{2}\left( \frac{\dim \Sigma}{2} - \dim(\Lambda_1 \cap \Pi) \right).
$$
\end{enumerate}
\end{lemma}

A similar invariant is the Kashiwara index of a triple of Lagrangian planes
\begin{definition}
The \emph{Kashiwara index} of the triple $(\Lambda_1,\Pi,\Lambda_2)$ is the signature of the form $q$:
$$
\Ki(\Lambda_1,\Pi,\Lambda_2) = \sign q.
$$
\end{definition}

\begin{lemma}
\label{lemm:kashiwara} 
The Kashiwara index has the following properties for all $\Lambda_i,\Pi \in L(\Sigma)$
\begin{enumerate}
\item Explicit finite bounds
$$
|\Ki(\Lambda_1,\Lambda_2,\Lambda_3)| \leq \frac{\dim \Sigma}{2};
$$
\item The cocyle property
$$
\Ki(\Lambda_2,\Lambda_3,\Lambda_4) - \Ki(\Lambda_1,\Lambda_3,\Lambda_4)  + \Ki(\Lambda_1,\Lambda_2,\Lambda_4) - \Ki(\Lambda_1,\Lambda_2,\Lambda_3) = 0 ;
$$
\item Antisymmetry
$$
\Ki(\Lambda_{p(1)},\Lambda_{p(2)},\Lambda_{p(3)}) = (-1)^{\sign (p)}\Ki(\Lambda_1,\Lambda_2,\Lambda_3),
$$
where $p$ is a permutation of $\{1,2,3\}$;
\item Relation with the positive Maslov index
$$
-\Ki(\Lambda_1,\Pi,\Lambda_2) + 2\ind_{\Pi}(\Lambda_1,\Lambda_2) + \dim (\Lambda_1 \cap \Lambda_2) = \frac{\dim(\Sigma)}{2}.
$$
\end{enumerate}
\end{lemma}
The proves of the first three properties can be found in~\cite[Section 1.4]{gosson} or~\cite{vergne}[Section 1.5]. The last one is proved in~\cite[Lemma 5]{agrachev_quadratic}.

Let us try to understand what are these indices geometrically in the simplest case, when $\Sigma = \R^2$ (see picture~\ref{fig:maslov}). Fix some Darboux coordinates $(p,q)$, s.t. $\Pi = \{(p,0)\}$. Then all the Lagrangian planes close to $\Pi$ are parametrized by a single parameter $S$ as $(p,S p)$. Consider a curve of Lagrangian planes $\Lambda(t): [-1,1]\to L(\Sigma)$ s.t. $\Lambda_0 = \Pi$. Then we easily compute, the derivative
$$
\dot \Lambda_0(\lambda) = \sigma\left( (p,0),(p,\dot S_0 p) \right) = \dot S_0 p^2.
$$
Thus when the curve $\Lambda(t)$ crosses $\Pi$ in the clockwise direction, we add $+1$ to the Maslov index, and $-1$ if it crosses clockwise.

Consider now $\ind_\Pi(\Lambda_{-1},\Lambda_1)$ and $\Ki(\Lambda_{-1},\Pi,\Lambda_1)$. By working out the definitions one can check that the values of both indices depend only on the relative positions of $\Lambda_{-1},\Pi,\Lambda_1$, where we have four situations, some of which are depicted in Figure~\ref{fig:maslov}:
\begin{enumerate}
\item if $\Lambda_{-1}=\Lambda_1 = \Pi$, then $\ind_\Pi(\Lambda_{-1},\Lambda_1) = \Ki(\Lambda_{-1},\Pi,\Lambda_1) = 0$;
\item if $\Lambda_{-1}=\Pi$ or $\Lambda_1 = \Pi$, then $\ind_\Pi(\Lambda_{-1},\Lambda_1) =1/2 $ and  $\Ki(\Lambda_{-1},\Pi,\Lambda_1) = 0$;
\item if by rotating $\Lambda_{-1}$ in the clockwise direction we meet $\Lambda_1$ before $\Pi$, then $$\ind_\Pi(\Lambda_{-1},\Lambda_1) = 0 \qquad \text{ and } \qquad \Ki(\Lambda_{-1},\Pi,\Lambda_1) = -1;$$
\item if by rotating $\Lambda_{-1}$ in the clockwise direction we meet $\Pi$ before $\Lambda_1$, then $$\ind_\Pi(\Lambda_{-1},\Lambda_1) = 1 \qquad \text{ and } \qquad \Ki(\Lambda_{-1},\Pi,\Lambda_1) = 1.$$
\end{enumerate}

The fact that these indices depend only on the relative positions of the Lagrangian planes is a consequence of the following statement.
\begin{proposition}[\cite{gosson}]
\label{prop:homotopy}
The Kashiwara index $\Ki(\Lambda_1,\Lambda_2,\Lambda_3)$ and the positive Maslov index $\ind_{\Lambda_2}(\Lambda_1,\Lambda_3)$ are constant on the set
$$
\{(\Lambda_1,\Lambda_2,\Lambda_3)\,:\, \dim (\Lambda_1 \cap \Lambda_2) =k_{1}, \dim (\Lambda_2 \cap \Lambda_3) =k_{2},\dim (\Lambda_3 \cap \Lambda_1) =k_{3}\}\subset L(\Sigma)^3,
$$
where $k_i$ are some constants.
\end{proposition}

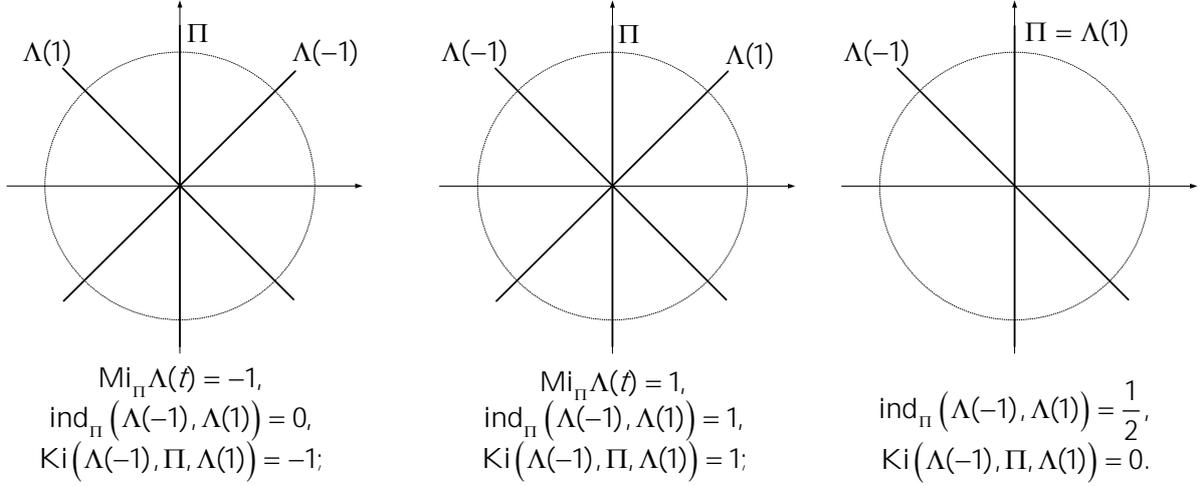
\begin{figure}[h]
\begin{center}

\resizebox{13.5cm}{6cm}{

\tikzset{every picture/.style={line width=0.75pt}} 

\begin{tikzpicture}[x=0.75pt,y=0.75pt,yscale=-1,xscale=1]

\draw  [dash pattern={on 0.84pt off 2.51pt}] (38.43,151.53) .. controls (38.43,112.45) and (70.73,80.76) .. (110.57,80.76) .. controls (150.41,80.76) and (182.71,112.45) .. (182.71,151.53) .. controls (182.71,190.62) and (150.41,222.3) .. (110.57,222.3) .. controls (70.73,222.3) and (38.43,190.62) .. (38.43,151.53) -- cycle ;
\draw    (109.51,40.65) -- (110.79,235.99) ;
\draw [shift={(109.49,37.65)}, rotate = 89.63] [fill={rgb, 255:red, 0; green, 0; blue, 0 }  ][line width=0.08]  [draw opacity=0] (10.72,-5.15) -- (0,0) -- (10.72,5.15) -- (7.12,0) -- cycle    ;
\draw    (196.5,151.24) -- (23.9,150.55) ;
\draw [shift={(199.5,151.25)}, rotate = 180.23] [fill={rgb, 255:red, 0; green, 0; blue, 0 }  ][line width=0.08]  [draw opacity=0] (10.72,-5.15) -- (0,0) -- (10.72,5.15) -- (7.12,0) -- cycle    ;
\draw [line width=1.5]    (39.94,219.94) -- (180.7,79.99) ;
\draw [line width=1.5]    (180.7,221.1) -- (39.72,79.28) ;
\draw [line width=1.5]    (109.49,54.23) -- (110.93,247.8) ;
\draw  [dash pattern={on 0.84pt off 2.51pt}] (248.76,151.53) .. controls (248.76,112.45) and (281.06,80.76) .. (320.91,80.76) .. controls (360.75,80.76) and (393.05,112.45) .. (393.05,151.53) .. controls (393.05,190.62) and (360.75,222.3) .. (320.91,222.3) .. controls (281.06,222.3) and (248.76,190.62) .. (248.76,151.53) -- cycle ;
\draw    (319.85,40.65) -- (321.12,235.99) ;
\draw [shift={(319.83,37.65)}, rotate = 89.63] [fill={rgb, 255:red, 0; green, 0; blue, 0 }  ][line width=0.08]  [draw opacity=0] (10.72,-5.15) -- (0,0) -- (10.72,5.15) -- (7.12,0) -- cycle    ;
\draw    (407,150.75) -- (234.23,150.55) ;
\draw [shift={(410,150.75)}, rotate = 180.07] [fill={rgb, 255:red, 0; green, 0; blue, 0 }  ][line width=0.08]  [draw opacity=0] (10.72,-5.15) -- (0,0) -- (10.72,5.15) -- (7.12,0) -- cycle    ;
\draw [line width=1.5]    (250.27,219.94) -- (391.03,79.99) ;
\draw [line width=1.5]    (391.03,221.1) -- (250.06,79.28) ;
\draw [line width=1.5]    (319.83,54.23) -- (321.26,247.8) ;
\draw  [dash pattern={on 0.84pt off 2.51pt}] (459.76,151.53) .. controls (459.76,112.45) and (492.06,80.76) .. (531.91,80.76) .. controls (571.75,80.76) and (604.05,112.45) .. (604.05,151.53) .. controls (604.05,190.62) and (571.75,222.3) .. (531.91,222.3) .. controls (492.06,222.3) and (459.76,190.62) .. (459.76,151.53) -- cycle ;
\draw    (530.85,40.65) -- (532.12,235.99) ;
\draw [shift={(530.83,37.65)}, rotate = 89.63] [fill={rgb, 255:red, 0; green, 0; blue, 0 }  ][line width=0.08]  [draw opacity=0] (10.72,-5.15) -- (0,0) -- (10.72,5.15) -- (7.12,0) -- cycle    ;
\draw    (618.3,150.75) -- (445.23,150.55) ;
\draw [shift={(621.3,150.75)}, rotate = 180.07] [fill={rgb, 255:red, 0; green, 0; blue, 0 }  ][line width=0.08]  [draw opacity=0] (10.72,-5.15) -- (0,0) -- (10.72,5.15) -- (7.12,0) -- cycle    ;
\draw [line width=1.5]    (602.03,221.1) -- (461.06,79.28) ;
\draw [line width=1.5]    (530.83,54.23) -- (532.26,247.8) ;

\draw (114.83,54.23) node [anchor=north west][inner sep=0.75pt]    {$\Pi $};
\draw (158.7,58.39) node [anchor=north west][inner sep=0.75pt]    {$\Lambda ( -1)$};
\draw (21.03,57.72) node [anchor=north west][inner sep=0.75pt]    {$\Lambda ( 1)$};
\draw (325.17,54.23) node [anchor=north west][inner sep=0.75pt]    {$\Pi $};
\draw (375.7,58.39) node [anchor=north west][inner sep=0.75pt]    {$\Lambda ( 1)$};
\draw (229.37,58.05) node [anchor=north west][inner sep=0.75pt]    {$\Lambda ( -1)$};
\draw (536.17,54.23) node [anchor=north west][inner sep=0.75pt]    {$\Pi =\Lambda ( 1)$};
\draw (440.37,58.05) node [anchor=north west][inner sep=0.75pt]    {$\Lambda ( -1)$};
\draw (55.83,256.57) node [anchor=north west][inner sep=0.75pt]    {$Mi_{\Pi } \Lambda \ =\ -1,$};
\draw (24.17,278.07) node [anchor=north west][inner sep=0.75pt]    {$ind_{\Pi }( \Lambda ( -1) ,\Lambda ( 1)) \ =\ 0,$};
\draw (13.17,300.07) node [anchor=north west][inner sep=0.75pt]    {$Ki( \Lambda ( -1) ,\Pi ,\Lambda ( 1)) \ =\ -1;$};
\draw (270.83,255.9) node [anchor=north west][inner sep=0.75pt]    {$Mi_{\Pi } \Lambda \ =\ 1,$};
\draw (233.17,277.4) node [anchor=north west][inner sep=0.75pt]    {$ind_{\Pi }( \Lambda ( -1) ,\Lambda ( 1)) \ =\ 1,$};
\draw (231.17,299.4) node [anchor=north west][inner sep=0.75pt]    {$Ki( \Lambda ( -1) ,\Pi ,\Lambda ( 1)) \ =\ 1;$};
\draw (434.17,268.4) node [anchor=north west][inner sep=0.75pt]    {$ind_{\Pi }( \Lambda ( -1) ,\Lambda ( 1)) \ =\ 1/2,$};
\draw (443.17,294.4) node [anchor=north west][inner sep=0.75pt]    {$Ki( \Lambda ( -1) ,\Pi ,\Lambda ( 1)) \ =\ 0.$};

\end{tikzpicture}

 }
\end{center} 
 
\caption{The Kashiwara and the positive Maslov indices in $\R^2$}
\label{fig:maslov}
\end{figure}

To state precisely what is the relation between the indices $\Mi$, $\Ki$ and $\ind$ we need the following definition
\begin{definition}
A curve $\Lambda(t)$ is called \emph{simple} if there exists $\Delta \in L(\Sigma)$, s.t. $\Lambda(t) \in \Delta^\pitchfork$, i.e. it is entirely contained in a single affine coordinate chart.
\end{definition}
\begin{proposition}[\cite{agrachev_quadratic}, Proposition A6]
Let $\Lambda(t)$, $t\in[0,1]$ be a continuous curve, s.t. there exists $\Delta \in L(\Sigma)$, for which $\Lambda(t) \cap \Delta = \Pi \cap \Delta = \{0\}$. Then
$$
\Mi_\Pi( \Lambda(t)) = \frac{1}{2}\left( \Ki(\Delta,\Lambda_0,\Pi) - \Ki(\Delta,\Lambda_1,\Pi) \right).
$$
\end{proposition} 

Any two given points $\Lambda_1,\Lambda_2 \in L(\Sigma)$ can be joined by a simple monotone curve. It is easy to see this using an affine chart on the Grassmanian. So it makes sense to reformulate this result for a closed monotone curve $\Lambda(t)$.
\begin{proposition}[\cite{agrachev_quadratic},Proposition 2]
\label{prop:maslov}
Suppose that $\Lambda(t)$, $t\in[0,1]$ is a closed continuous monotone curve, $0 =t_0 < t_1 < t_2 <... < t_N  = 1$ is a partition of $[0,1]$ and $\Lambda_i = \Lambda(t_i)$. Then one has the estimate
$$
\Mi(\Lambda(t)) \geq \sum_{i= 0}^N \ind_\Pi(\Lambda_i,\Lambda_{i+1}) ,
$$
where $\Lambda_{N+1} = \Lambda_0$. Moreover if all pieces $\Lambda(t)|_{[t_{i},t_{i+1}]}$ are simple, i.e. there exist $\Delta_i\in L(\Sigma)$, s.t. $\Delta \cap \Lambda(t)|_{[t_i,t_{i+1}]} = \{0\}$, then we have an equality
$$
\Mi(\Lambda(t)) = \sum_{i= 0}^N \ind_\Pi(\Lambda_i,\Lambda_{i+1}) = \frac{1}{2}\sum_{i= 0}^N \left( \Ki(\Delta_i, \Lambda_{i}, \Pi) -  \Ki(\Delta_i, \Lambda_{i+1}, \Pi)\right).
$$

\end{proposition}

This motivates the following definition, that extends the notions of Maslov index and monotonicity from continuous curves to general curves in the Lagrangian Grassmanian. This extension is important, since even in the relatively simple case of bang-bang trajectories the Jacobi curves are discontinuous.
\begin{definition}
\label{def:monotone2}
Let $\Lambda(t) : [0,T] \to L(\Sigma)$ be a curve in the Lagrangian Grassmanian. Given a partition $D = \{0=t_0 < t_1 <...< t_N = T\}$ we define 
$$
\ind^D_\Pi \Lambda(t) = \sum_{i=0}^{N-1} \ind_{\Pi}(\Lambda_i,\Lambda_{i+1}).
$$
where $\Lambda_i $ are as in the Proposition~\ref{prop:maslov}. We say that $\Lambda(t)$ is \emph{monotone increasing}, if
$$
\ind_\Pi \Lambda(t) = \sup_D \ind^D_\Pi \Lambda(t) < + \infty.
$$
The quantity $\ind_\Pi \Lambda(t)$ we call the \emph{Maslov index of a monotone curve}.
\end{definition}

The Maslov index defined in such way inherits many useful properties of the usual Maslov index defined as an intersection number. For example, we have the following lemma.
\begin{lemma}
\label{lem:maslov_ind}
Let $\Lambda(t)$ be a closed monotone curve in the sense of Definition~\ref{def:monotone2}. Then 
$$
\ind_\Pi \Lambda(t) = \ind_\Delta \Lambda(t), \qquad \forall \Delta,\Pi \in L(\Sigma).
$$
\end{lemma}

\begin{proof}
Since the curve is monotone, the supremum in the definition is finite. But since it can take only discrete values, it must be attained by some partition $D$, i.e.
$$
\ind_\Pi \Lambda(t) = \sum_{i=0}^N \ind_\Pi(\Lambda(t_i),\Lambda(t_{i+1})),
$$
where $\Lambda(t_{N+1}) = \Lambda(t_0)$, $t_i \in D$. At the same time we can join $\Lambda(t_i)$ with simple monotone curves and construct this way a closed curve $\hat \Lambda(t)$. Then by the Proposition~\ref{prop:maslov_closed}
$$
\ind_\Pi \Lambda(t) = \Mi_{\Pi} (\hat \Lambda(t)) = \Mi_{\Delta} (\hat \Lambda(t)) = \sum_{i=1}^n \ind_{\Delta}(\Lambda(t_i),\Lambda(t_{i+1})) = \ind_\Delta \Lambda(t).
$$
\end{proof}

However we would also like to see that this definition is well defined, i.e. that it coincides with the previous one in the case of differentiable curves.

\begin{theorem}
\label{thm:monot}
A differentiable curve $\Lambda(t) \in L(\Sigma)$ is monotone increasing if, and only if, $\ind_\Pi \Lambda(t) < +\infty$ for some $\Pi\in L(\Sigma)$. 
\end{theorem} 

\begin{proof}

Let $\Lambda(t)$ be $C^1$ on $[0,T]$ and monotone in the sense of Definition~\ref{def:monotone}. Let us assume by contradiction that it is not monotone in the sense of Definition~\ref{def:monotone2}. Then there exists a series of splittings $D_n$ of the form $0 = t_0^n < ... < t^n_{k_n}$, s.t. $\ind_{\Pi}^{D_n} \Lambda_{t} \to +\infty$ as $n\to \infty$. By refining $D_n$ if necessary, we can always assume that each restriction $\Lambda|_{[t_i^n,t_{i+1}^n]}$ is simple. Then by Proposition~\ref{prop:maslov} we have $\Mi_{\Pi} \Lambda(t) = \infty$.

We claim that this is impossible. Indeed in~\cite{agr_gamk_symp} it was shown, that one can introduce a Riemannian metric on $L(\Sigma)$, s.t. the length of the monotone curve $\Lambda(t)$ can be bounded by the Maslov index as
$$
\frac{\pi}{2\sqrt{n}} \Mi_\Pi \Lambda(t) \leq length(\Lambda(t)) \leq \frac{\pi}{2} \Mi_\Pi \Lambda(t).
$$ 
Thus the length must be infinite and we arrive at a contradiction with the fact that the curve is $C^1$. 

Now we prove the converse. Suppose that we have $\Lambda(t) \in C^1([0,T],L(\Sigma))$ with $\ind_\Pi \Lambda(t) < + \infty$. We claim that $\Lambda(t)$ is monotone increasing in the sense of Definition~\ref{def:monotone}. Let us prove the opposite statement that for a non-increasing curve $\Lambda(t)$ the index $\ind_\Pi \Lambda(t) $ is infinite. 

By Lemma~\ref{lem:maslov_ind} and the first property in Lemma~\ref{lemm:maslov}, we can see that for any Lagrangian plane $\Delta$ 
$$
|\ind_{\Pi} \Lambda(t) - \ind_{\Delta} \Lambda(t)|\leq n.
$$ 
So it is enough to prove that $\ind_{\Delta}\Lambda|_{[t_1,t_2]} = +\infty$ for any $\Delta$ of our choice.

If $\Lambda(t)$ is not monotone increasing in the sense of Definition~\ref{def:monotone}, then there exists a moment of time $t\in[0,T]$ and a vector $\lambda$, s.t. $\dot{\Lambda}(t)(\lambda)<0$. We can choose a sufficiently small subinterval $[t_1,t_2]\in[0,T]$ such that it contains the moment of time $t$ and the restriction of $\Lambda(t)$ is simple, i.e. lies entirely in $\Delta^\pitchfork$. 

Let us consider the coordinate chart $\Delta^\pitchfork$. We can find $[t_1,t_2]$ sufficiently small, so that $S(t_2) - S(t_1) \ngeq 0$ (or else the curve would have been monotone). Let us take any smooth simple monotone curve $\alpha:[0,T]\to L(\Sigma)$ that joins $\Lambda(t_1)$ with $\Lambda(t_2)$. The curve $\alpha(t)$ can not lie entirely in $\Delta^\pitchfork$, because in this case we would have had by monotonicity $S(t_2) - S(t_1) = S^\alpha(T)-S^\alpha(0) \geq 0$. Therefore at some point $\alpha$ must intersect the Maslov train $\cM_{\Delta}$ and $\Mi_{\Delta} \alpha(t) \geq 1$. Then by Proposition~\ref{prop:maslov} we have $\ind_{\Delta}(\Lambda(t_1),\Lambda(t_2))\geq 1$ as well. But the same is true for any other subinterval of $[t_1,t_2]$. Thus by splitting it into smaller subintervals, we will find that $\ind^D_{\Delta} \Lambda(t) \to +\infty$.
\end{proof}

Although these invariants were already successfully applied in~\cite{agrachev_quadratic,agr_gamk_symp} to the study of the second variation of some classes of optimal control problem, in order to formulate the main Morse theorem we need one more symplectic invariant.

\begin{definition}
Let $\widetilde{L(\Sigma)}$ be the universal covering of $L(\Sigma)$. The Leray index is the unique mapping 
$$
\Li: \widetilde{L(\Sigma)} \times \widetilde{L(\Sigma)} \to \Z
$$
that satisfies the following two properties:
\begin{enumerate}
\item $\Li$ is locally constant on the set $\{(\tilde\Lambda_1,\tilde\Lambda_2): \Lambda_1 \cap \Lambda_2 = \{0\}\}$;
\item $-\Li(\tilde\Lambda_2,\tilde\Lambda_3)+\Li(\tilde\Lambda_1,\tilde\Lambda_3) - \Li(\tilde\Lambda_1,\tilde\Lambda_2) = \Ki(\Lambda_1,\Lambda_2,\Lambda_3)$.
\end{enumerate}
\end{definition} 

An explicit construction of the Leray index using matrix logarithms can be found in~\cite[Section 3.2.3]{gosson} or~\cite{sternberg}. We only list its main properties, that are going to be useful for the computations. Note however that this Leray index is the minus Leray index introduced in the two cited references.

\begin{lemma}
\label{lemm:leray} 
The Leray index $Li$ has the following properties
\begin{enumerate}
\item Antisymmetry
$$
\Li(\tilde\Lambda_1,\tilde\Lambda_2) = -\Li(\tilde\Lambda_2,\tilde\Lambda_1),
$$
\item If $\tilde\Lambda(t)$ as a lift a closed continuous curve $\Lambda(t) :[0,T] \to L(\Sigma)$ to $\widetilde{L(\Sigma)}$, then

$$
\Li(\tilde\Lambda(T),\tilde\Lambda) - \Li(\tilde\Lambda(0),\tilde\Lambda) = 2\Mi(\Lambda(t)), \qquad \forall \tilde \Lambda \in \widetilde{L(\Sigma)}.
$$
\end{enumerate}
\end{lemma}

The Leray index allows to define the Maslov index and other intersection indices for curves in the Lagrangian Grassmanian and symplectic group in an abstract way. But one of its most important applications is that it can be used to construct an explicit model for the universal covering space $\widetilde{L(\Sigma)}$.

\begin{theorem}
\label{thm:univer_cover}
Let $\tilde\Lambda_\alpha$ be a lift of an arbitrary Lagrangian plane $\Lambda_\alpha$ to the universal covering $\widetilde{L(\Sigma)}$. Define a mapping $\Phi_\alpha: \widetilde{L(\Sigma)} \to L(\Sigma) \times \Z$ by
$$
\Phi_\alpha(\tilde\Lambda) = \left( \Lambda, \frac{1}{2}\Li(\tilde\Lambda,\tilde\Lambda_\alpha) \right).
$$  
Then
\begin{enumerate}
\item The mapping $\Phi_\alpha$ is a bijection, whose restrictions to the subset $\{\tilde{\Lambda} \in \widetilde{L(\Sigma)}: \Lambda\cap \Lambda_\alpha = \{0\}\}$ is a homeomorphism onto $\{\Lambda \in L(\Sigma): \Lambda\cap \Lambda_\alpha = \{0\}\}$.
\item The set of all bijections $\Phi_\alpha$ forms a system of local charts of $\widetilde{L(\Sigma)}$ whose transitions $\Phi_{\alpha\beta} = \Phi_\alpha \Phi_\beta^{-1}$ are the functions
$$
\Phi_{\alpha\beta}(\Lambda,k) = \left( \Lambda,k + \frac{\Ki(\Lambda,\Lambda_\alpha,\Lambda_\beta) - \Li(\tilde\Lambda_\alpha,\tilde{\Lambda}_\beta)}{2} \right)
$$
\end{enumerate}
\end{theorem}

The proof of this theorem and the last lemma, as well as many other applications of the Leray index can be found in~\cite{gosson}[Theorem 4.5].

We need two lemmas related to curves in the Lagrangian Grassmanian and its universal covering. We will use them only to prove the main Morse Theorem~\ref{thm:main_index_result}. So we just sketch the proofs.

\begin{lemma}
\label{lem:count_int}
For any countable set $S \in L(\Sigma)$ the set of Lagrangian planes, that intersect transversally any Lagrangian plane from $S$, is dense in $L(\Sigma)$.
\end{lemma}

\begin{proof}
The proof is a simple consequence of the Baire category theorem. Indeed, we know that $\Lambda^\pitchfork$ is an open set in $L(\Sigma)$. One should just prove that those sets are dense, then the intersection
$$
\bigcap_{\Lambda \in S} \Lambda^\pitchfork 
$$
must be dense and therefore non-empty~\cite[Remark 2.5.18]{piccione_tausk}.
\end{proof}

\begin{lemma}
\label{lem:homotop}
Any two simple monotone curves connecting $a,b \in L(\Sigma)$ are homotopic. 
\end{lemma}

This lemma is a direct consequence of Propositions~\ref{prop:homotopy} and~\ref{prop:maslov}. They show that Maslov index of monotone curves depends only on the relative position of its end-points.

\begin{remark}
\label{rem:lift}
This result has an important application that we will use later. Let $\Lambda(t)$ be a curve in $L(\Sigma)$ with a finite number of discontinuities. Then there is a canonical way of lifting the curve to the universal covering $\widetilde{L(\Sigma)}$. One has to glue all the discontinuities with simple monotone curves and lift it to the universal covering and then delete the lifts of the glued in monotone parts. The result will not depend on the way of gluing. Indeed, the previous lemma shows that two monotone curves are homotopic and therefore their lifts starting at the same point will also end at the same point. 
\end{remark}

\subsection{Morse-type theory}
\label{sec:sympl_morse}

Now we are ready to state and prove Morse-type theorems. The simplest one allows us to compute the dimension of the kernel of the Hessian. This is Lemma~\ref{lem:ker} from Appendix~\ref{app:existence}.

The next step is to extract the information about the index of the Hessian from the Jacobi curve. A simple theorem of such type is Theorem~\ref{thm:index_add} that is used in the proof of existence and uniqueness. But it is clear that in general one can not replace inequality in the statement by an equality. The right hand side of~\eqref{eq:ineq_theorem_long_name} is limited by the dimension of the manifold $M$, while the jump in the index can be arbitrary large. Nevertheless, when we take piecewise constant functions we can reconstruct exact formulas. The idea is that when we add some constant variations, using our algorithm (see Theorem~\ref{thm:algorithm}) we can track exactly how the $\cL$-derivative changes and use this to obtain an exact formula for the index. This will be the main building block in the general Morse index theorem, that we will prove immediately after.

We will rely heavily on the following two lemmas from linear algebra
\begin{lemma}
\label{lem:important}
Suppose that $Q$ is a quadratic form defined on $\R^N$ and let $V\subset \R^N$ be some subspace. If we define
$$
V^\perp = \{x \in \R^N \;:\; Q(x,y)= 0,\forall y \in V )\},
$$
then
\begin{equation}
\label{eq:index_formula}
\ind^+ Q = \ind^+ Q|_V + \ind^+ Q|_{V^\perp} + \dim(V\cap V^\perp) - \dim(V\cap \ker Q)
\end{equation}
\end{lemma}

\begin{lemma}
\label{lem:orthogonal}
Let $Q:V_2 \times V_2 \to \R$ be a quadratic form defined on a finite-dimensional space $V_2$, $A: V_2 \to \R^n$ be a linear map and $N\subset \R^n$ be a linear subspace. Take any subspace $V_1 \subset V_2$ and write $V_i^N = V_i \cap A^{-1}(N \cap \im A)$. Then the orthogonal complement of $V_1^N$ in $V_2^N$ with respect to $Q$ consists of vectors $v\in V_2^N$, for which there exists $\tilde \xi$ in the annihilator $N^\perp \subset (\R^n)^*$, s.t.
\begin{equation}
\label{eq:usual}
\langle \tilde \xi, Aw\rangle + Q(v,w) = 0, \qquad \forall w\in V_1. 
\end{equation}

Similarly $\ker Q\cap V_1^N$ consists of vectors $v\in V_1^N$, for which there exists $\tilde \xi$ in the annihilator $N^\perp \subset (\R^n)^*$, s.t. the equality above holds for all $w\in V_2$.

\end{lemma}

We start with the computation of the Morse index of the Hessian restricted to piecewise constant variations.

\begin{theorem}
\label{thm:morse_apr}
Let $D = \{0 = t_0 < t_1 < ... < t_N  = t\}$ be a partition of the interval $[0,t]$ and let $V_D$ be a direct product of the space of variations of the initial point and piece-wice constant variations with jumps at moments of time $t_i$. We denote by $V_i \subset V_D$ the subspace of $V_D$ of variations that are zero for $t > t_i $ and $V_i^0 = V_i \cap \ker dE_{N_0,t}[\tilde{\omega}]$. 
Then the following formula is true
\begin{equation}
\label{eq:maslov_index}
\ind^- \Hess (E_{N_0,t},\nu J_t)[\tilde \omega, \lambda(t)]|_{V_D^0} = \sum_{i=-1}^{N} \ind_\Pi(\cL_i, \cL_{i+1}) + \dim\left( \bigcap_{i=-1}^N \cL_i\right) - n,
\end{equation}
where for simplicity we wrote $\cL_i = \cL (E_{N_0,t},\nu J_t)[\tilde \omega, \lambda(t)]|_{\Omega_{N_0}^t \cap V_i}$, $\cL_{-1} = \cL_{N+1} = \Pi$ and $\cL_0 = T_{\tilde{q}(0)}^\perp N_0 \times T_{\tilde{q}(0)} N_0$.
\end{theorem}

\begin{proof}
As before we write
$$
Q= \lambda d^2 E_{N_0,t}[\tilde \omega] - \nu d^2J_t[\tilde \omega].
$$
As we have already mentioned, $\Hess (E_{N_0,t},\nu J_t)[\tilde \omega, \lambda(t)]$ is equal to $-Q|_{\ker dE_{N_0,t}}$ as a quadratic form. So it is enough to prove the formula with $\ind^+ Q|_{V_D^0}$ on the left-hand side.

We prove it by a recursive computation of $\ind^+ Q|_{V_{i+1}^0}$ in terms of $\cL_m$, $m\leq i$. The main tool will be the formula (\ref{eq:index_formula}). We denote by $Q_i$ the restriction of $Q$ to $V_i^0$ and by $(V_i^0)^\perp$ the orthogonal complement of $V_i^0$ with respect to $Q_{i+1}$. First we establish the formula for $\ind^+ {Q_{i+1}|_{(V_i^0)^\perp}}$ and then for $\dim(V_i^0 \cap (V_i^0)^\perp) - \dim(\ker Q_{i+1} \cap V_i^0)$ in terms of $\cL_i$.

\emph{Step 1.} We prove the following statement. Given two subspaces $U_1,U_2$ such that
$$
T_{\tilde{q}(0)}N_0 \subset U_1 \subset U_2 \subset T_{\tilde{q}(0)}N_0\times L^2_k[0,t] \approx \Omega_{N_0}^t,
$$ 
we claim that the subspace $(U_1^{0})^\perp$ is equal to a subspace $W_2\subset U^0_2$ which consists of $(\zeta, v_2(\tau)) \in U_2^0$, s.t. there exists $\eta \in \cL_0$ for which the following conditions are satisfied
$$
\pi (\eta) = \zeta,
$$
\begin{equation}
\label{eq:def_w2}
\int_0^t  \sigma \left( \eta + \int_0^\tau X(\theta) v_2(\theta) d\theta, X(\tau) v_1(\tau) \right) + b(\tau)(v_2(\tau),v_1(\tau)) d\tau = 0,\quad  \forall v_1(\tau)\in U_1 
\end{equation}

This is a consequence of Lemma~\ref{lem:orthogonal}. Indeed, as we have discussed in Section~\ref{sec:optim} vector fields $X(\tau)$ are lifts of $g'(\tau)$. Therefore from formula~\eqref{eq:first_der} in Appendix~\ref{app:chrono}, which gives an explicit form for the first variation, it follows that we can characterize the kernel of $d E_{N_0,t}[\tilde{\omega}]$ in terms of the Hamiltonian vector field $X(\tau)$ as
$$
\ker d E_{N_0,t}[\tilde{\omega}]= \left\{(v(\tau),\zeta) \in \Omega_{N_0}^t \;: \; \eta+\int_0^t X(\tau) v(\tau) d\tau \in \Pi, \forall \eta \in \cL_0, \pi(\eta)=\zeta\right\}.
$$
We apply Lemma~\ref{lem:orthogonal} with $A$ being equal to the operator
$$
A:(\zeta,v)\mapsto \int_0^t X(\tau) v(\tau) d\tau
$$
and $N = \Pi \times T_{\tilde{q}(0)} N_0 \subset T_{\lambda(0)}(T^*M)$. Then $\tilde \xi \in (T_{\lambda(0)}(T^*M))^*$ from Lemma~\ref{lem:orthogonal} must annihilate $N$. Since $\sigma$ is a non-degenerate symplectic form we can use it to identify $(T_{\lambda(0)}(T^*M))^*$ with $(T_{\lambda(0)}(TM))$ and in this case the corresponding vector $\xi \in N^\perp \simeq N^\angle \simeq T_{q_0}^\perp N_0$ and hence
$$
\langle \tilde\xi, Av_1 \rangle = \sigma \left( \xi,\int_0^t X_\tau v_1(\tau)d\tau \right)
$$
from which the statement follows.

\emph{Step 2.} We have by definition and formula~\eqref{eq:sec_deriv} that the subspace $(V_i^{0})^\perp$ is equal to the space of variations $(\zeta,v(\tau), \alpha) \in V_{i+1}^0$, s.t.
$$
\int_0^{t_{i}} \sigma \left(\zeta + \int_0^\tau X(\theta) v(\theta) d\theta , X(\tau) w(\tau)\right) + b(\tau) (v(\tau),w(\tau)) d\tau = 0, \quad \forall w(\tau) \in V_i^0.
$$
Then from the step 1 it follows that $(V_i^{0})^\perp$ is actually equal to the space $W_i$ of vectors $(\zeta,v(\tau), \alpha) \in V_i \times \R^k$, s.t. there exists $\eta \in \cL_0$ for which $\pi(\eta)=\zeta$ and
\begin{equation}
\label{eq:def_w2_new}
\int_0^{t_i} \sigma \left( \eta + \int_0^\tau X(\theta) v(\theta) d\theta, X(\tau) w(\tau) \right) + b(\tau)(v(\tau),w(\tau)) d\tau = 0, \qquad \forall w(\tau) \in V_i,
\end{equation}
\begin{equation}
\label{eq:some_eq}
\eta + \int_0^{t_i} X(\tau) v(\tau) d\tau + \int_{t_i}^{t_{i+1}} X(\tau) d\tau \cdot \alpha\ \in \Pi.
\end{equation}
We denote by
$$
\lambda = \eta + \int_0^{t_i} X(\tau) v(\tau) d\tau.
$$
Then the first condition just tells us that $\lambda\in \cL_i$. 

Let $(\zeta,v(\tau),\alpha) \in W_i$ and $\eta = (\mu,\zeta)\in T^\perp_{\tilde{q}(0)}N_0\times T_{\tilde{q}(0)}N_0$, then we obtain
\begin{align*}
Q_{i+1}|_{V_i}(\zeta,v(\tau),\alpha)&= \int_0^{t_i} \sigma\left(\zeta + \int_0^\tau X(\theta) v(\theta) d\theta, X(\tau) v(\tau) \right) + b(\tau)(v(\tau),v(\tau)) d\tau + \\
&+\int_{t_i}^{t_{i+1}} \sigma\left(\zeta + \int_0^{t_i} X(\theta) v(\theta) d\theta + \int_{t_i}^\tau X(\theta) d\theta \cdot \alpha, X(\tau) \alpha \right) + b(\tau)(\alpha,\alpha) d\tau = \\
&= -\sigma\left( \mu, \int_0^{t_i} X(\tau) v(\tau) d\tau  \right) + \\
&+\int_{t_i}^{t_{i+1}} \sigma\left(\zeta + \int_0^{t_i} X(\theta) v(\theta) d\theta + \int_{t_i}^\tau X(\theta) d\theta \cdot \alpha, X(\tau) \alpha \right) + b(\tau)(\alpha,\alpha) d\tau =\\
&=  \int_{t_i}^{t_{i+1}} \sigma\left( \lambda + \int_{t_i}^\tau X(\theta) d\theta \cdot \alpha, X(\tau) \alpha \right) + b(\tau)(\alpha,\alpha) d\tau
\end{align*}
where in the last equality we have used a consequence of (\ref{eq:some_eq})
$$
-\sigma\left( \mu, \int_0^{t_i} X(\tau) v(\tau) d\tau \right) = -\sigma\left( \mu, \eta+ \int_0^{t_i} X(\tau) v(\tau) d\tau \right) = \sigma\left( \mu, \int_{t_i}^{t_{i+1}} X(\tau) d\tau \cdot \alpha\right)
$$
since $\mu,\eta \in \cL_0$.

Thus we have shown that $Q_{i+1}|_{W_i}$ is equal to the form
$$
P(\lambda,\alpha) = \int_{t_i}^{t_{i+1}} \sigma\left( \lambda + \int_{t_i}^\tau X(\theta) d\theta \cdot \alpha, X(\tau) \alpha \right) + b(\tau)(\alpha,\alpha) d\tau
$$
defined on a finite-dimensional space
$$
S = \left\{(\lambda,\alpha) \in \cL_i\times \R^k \; : \;\lambda + \int_{t_i}^{t_{i+1}} X(\tau) d\tau \cdot \alpha \in \Pi  \right \}.
$$

Now we consider the quadratic form $q$ from the definition of the positive Maslov index defined on $(\cL_i + \cL_{i+1})\cap \Pi$. Let
\begin{equation}
\label{eq:subspace}
\lambda_1,\lambda_2 \in \cL_i, \qquad \lambda_2 + \int_{t_i}^{t_{i+1}} X(\tau) d\tau \cdot \alpha \in \cL_{i+1}, \qquad \lambda_1 + \lambda_2 + \int_{t_i}^{t_{i+1}} X(\tau) d\tau \in \Pi.
\end{equation}
Then from the definition of $\cL_{i+1}$ we obtain
\begin{align*}
q(\lambda_1,(\lambda_2,\alpha)) &= \sigma\left( \lambda_1, \lambda_2 + \int_{t_i}^{t_{i+1}} X(\tau) d\tau \cdot \alpha  \right) = \sigma\left( \lambda_1 + \lambda_2 - \lambda_2, \lambda_2 + \int_{t_i}^{t_{i+1}}X(\tau)d\tau \cdot \alpha \right) =\\
&=\sigma\left( \lambda_1 + \lambda_2, \int_{t_i}^{t_{i+1}} X(\tau) d\tau \cdot \alpha  \right) - \sigma\left( \lambda_2, \int_{t_i}^{t_{i+1}} X(\tau) d\tau \cdot \alpha  \right) = \\
&= \int_{t_i}^{t_{i+1}} \sigma\left( \lambda_1 + \lambda_2 + \int_{t_i}^{\tau} X(\theta) d\theta \cdot \alpha, X(\tau) \alpha \right) + b(\tau)(\alpha,\alpha)d\tau.
\end{align*}
If we denote
$$
\tilde{S} = \{(\lambda_1 + \lambda_2,\alpha) \in \cL_i\times \R^k \; : \; \lambda_j,\alpha \textit{ satisfy (\ref{eq:subspace})}\},
$$
then we have that $\tilde{S}\subset S$ and $P|_{\tilde{S}} = q|_{\tilde{S}}$. So
\begin{equation*}
\ind^+ Q_{i+1}|_{(V_i^0)^\perp} \geq \ind^+ q.
\end{equation*}

\emph{Step 3.} In order to prove the previous inequality we could have applied directly Theorem~\ref{thm:index_add} from the Appendix~\ref{app:existence}. However in order to prove the other inequality, we need the specific expressions for $Q_{i+1}|_{W_i}$ and $q$ obtained earlier. We want to show, that all $\lambda,\alpha$, that actually give a contribution to the index of $P$, lie in $\tilde{S}$. And indeed, this is just a consequence of our algorithm.

Take $(\lambda,\alpha) \in S$. In Theorem~\ref{thm:algorithm} we have defined subspaces $L\subset \cL_i$ and $E\subset \R^k$, which from their definition do not give contributions to the $\cL$-derivative and can be seen to lie in the kernel of $P$. Thus it is enough to consider $P$ on any complementary subspace $L^\perp$ and $E^\perp$. But from Theorem~\ref{thm:algorithm} we know, that for any $\alpha \in E^\perp$ there exists a unique $\lambda_2 \in L^\perp \subset \cL_i$, s.t.
$$
\lambda_2 + \int_{t_i}^{t_{i+1}} X(\tau) d\tau \cdot \alpha \in \cL_{i+1}.
$$
Thus we can take $\lambda_1 = \lambda - \lambda_2$ and then it follows that $(\lambda,\alpha) \in \tilde{S}$, which proves that $S\subset \tilde S$ and
\begin{equation*}
\ind^+ Q_{i+1}|_{(V_i^0)^\perp} \leq \ind^+ q.
\end{equation*}
Then from the definition of the positive Maslov index, we have
\begin{align}
\label{eq:step1}
\ind^+ Q_{i+1}|_{(V_i^0)^\perp} = \ind_\Pi(\cL_i,\cL_{i+1}) &- \frac{1}{2}\left( \dim(\cL_i\cap \Pi) + \dim(\cL_{i+1}\cap \Pi) \right) \\
&+ \dim\left( \cL_i\cap \cL_{i+1}\cap \Pi \right)\nonumber
\end{align}

Using exactly the same arguments one can prove the formula for $Q_1$, that gives the base of the induction
\begin{align}
\ind^+ Q_1 &= \ind_\Pi(\cL_0,\cL_1) - \frac{1}{2}\left( \dim(\cL_0\cap \Pi) + \dim(\cL_{1}\cap \Pi) \right) + \dim\left( \Pi\cap \cL_0 \cap \cL_{1} \right).  \label{eq:step2}
\end{align}

\emph{Step 4}. Now we obtain an expression for $\dim (V_i^0 \cap (V_i^0)^\perp) - \dim (\ker Q_{i+1} \cap V_i^0)$ in terms of $\cL_m$, $m\leq i+1$. Here again our algorithm plays the central role. It gives us a sequence of maps $P_i$
$$
T_{\tilde{q}(0)}^\perp N_0 \times T_{\tilde{q}(0)}N_0 = \cL_0 \xrightarrow{P_0} \cL_1 \xrightarrow{P_1} ... \xrightarrow{P_{N-1}} \cL_N = \cL_t(V_D).  
$$
Using Lemma~\ref{lem:orthogonal} like in step 2 we can show that $v\in V_i$ belongs to $\in V_i^0 \cap (V_i^0)^\perp$ if, and only if, there exists $\eta \in \cL_0$ such that
$$
\eta + \int_0^t X(\tau) v(\tau) d\tau \in \cL_i \cap \Pi.
$$
Thus we need to see how many variations $v$ correspond to a fixed $\lambda\in \cL_i \cap \Pi$, and we can do it by inverting $P_i$ and going backwards from $\cL_N$ to $\cL_{N-1}$, then to $\cL_{N-2}$ and so on. Maps $P_i$ are indeed invertible, since they are surjective linear maps between spaces of the same dimension. Thus by fixing $\lambda \in \cL_i \cap \Pi$, we get a sequence $P_{i-1}^{-1}(\lambda)$, $P_{i-2}^{-1}\circ P_{i-1}^{-1}(\lambda)$ and so on, that can be seen as a sort of a solution of the Jacobi equation passing through $\lambda$. 

Let $L_m \in \cL_m$, $E_m \in \R^k$ be the subspace $L$, $E$ from the Theorem~\ref{thm:algorithm} for $V = V_m$ and $L^\perp_m$, $E^\perp_m$ be the orthogonal complements in $\cL_m$ and $\R^k$ correspondingly. Note that to each $\lambda \in L^\perp_m$ corresponds a unique variation $\alpha \chi_{[t_m,t_{m+1}]}$, $\alpha \in E^\perp_m$. But if $\beta \in E_m$ the variation $(\alpha+\beta)\chi_{[t_m,t_{m+1}]}$ corresponds to the same vector in the $\cL$-derivative. Therefore we have $\sum \dim E_m$ of variations that correspond to the same $\lambda \in \cL_i \cap \Pi$. However all vectors
$$
\lambda \in \bigcap_{i=0}^m \cL_i \cap \Pi = \bigcap_{i=-1}^m \cL_i 
$$
correspond to the same variation $(\zeta, v) \equiv (0,0)$. Thus we obtain the following formula

\begin{equation}
\label{eq:step3}
\dim (V_i^0 \cap (V_i^0)^\perp) = \dim (\cL_i\cap \Pi) + \sum_{m=0}^{i-1} \dim E_m - \dim \left( \bigcap_{m=-1}^i \cL_m \right).
\end{equation} 

Now we compute $\dim (\ker Q_{i+1} \cap V_i^0)$. Using Lemma~\ref{lem:orthogonal}, the same proof as in the step 1 shows, that $v\in \ker Q_{i+1} \cap V_i^0$ if, and only if, there exists $\eta\in \cL_0$ for which $\pi(\eta) = \zeta$, s.t.
$$
\int_0^{t_i} \sigma \left( \eta + \int_0^\tau X(\theta) v(\theta) d\theta, X(\tau) w(\tau) \right) + b(\tau)(v(\tau),w(\tau)) d\tau = 0, \qquad \forall w(\tau) \in V_i,
$$
$$
\sigma \left( \eta + \int_0^{t_i}X(\tau) v(\tau) d\tau, \int_{t_i}^{t_{i+1}} X(\tau) d\tau \cdot \alpha \right)=0, \qquad \forall \alpha \in \R^k,
$$
$$
\eta + \int_0^{t_i} X(\tau) v(\tau) d\tau  \in \Pi.
$$
If denote 
$$
\lambda = \eta + \int_0^{t_i} X(\tau) v(\tau) d\tau,
$$
then equivalently we can write $\lambda\in \cL_i\cap \cL_{i+1}\cap \Pi$.

Using same argument as for $\dim (V_i^0 \cap (V_i^0)^\perp)$ we get
\begin{equation}
\label{eq:step4}
\dim (\ker Q_{i+1} \cap V_i^0) = \dim (\cL_i \cap \cL_{i+1} \cap \Pi) + \sum_{m=0}^{i-1} \dim E_m - \dim \left( \bigcap_{m=-1}^{i+1} \cL_m \right).
\end{equation}

So we sum over all $i$ the formulas (\ref{eq:step1})-(\ref{eq:step4}) to obtain
$$
\ind^+ Q|_{V_D} = \sum_{i=0}^{N-1} \ind_\Pi(\cL_i\cap \cL_{i+1}) - \frac{1}{2}\dim(\cL_0 \cap \Pi) - \frac{1}{2}\dim(\Pi\cap \cL_N) + \dim \left( \bigcap_{i=-1}^N \cL_i\right).
$$
The final formula follows from $\cL_{-1}=\cL_{N+1} = \Pi$ and property 4 in Lemma~\ref{lemm:maslov}.

\end{proof}

This approximation lemma can now be used to prove a very general Morse theorem, that establishes relation between some symplectic invariants of the Jacobi curve and index of the Hessian. After fixing some partition $D$, we introduce the following curves using the notations of the previous theorem
$$
\Lambda^D(t) = 
\begin{cases}
\Pi & \text{ if } -1\leq t <  0 = t_0,\\ 
\cL_i & \text{ if } t_i<t\leq t_{i+1},\\
\Pi & \text{ if } t_N<t\leq t_N + 1.
\end{cases}
$$

We extend the Jacobi curve $\cL_t$ by assuming that $\cL_t = \Pi$ for $t\in[-1,0)\cup(T,T+1]$. Then by definition $\Lambda^D(t) \to \cL_t$ pointwise as a generalized limit. To shorten the notations we also write $\Sigma = T_{\lambda(0)}(T^*M)$.

\begin{theorem}
\label{thm:main_index_result}
Suppose that $\ind^- \Hess (E_{N_0,t},\nu J_t) [\tilde{\omega},\lambda(t) ]<\infty$ at a Lagrange point $(\tilde{\omega},\lambda(t))$. Let $\tilde{\Pi}$ be a point in the universal covering $\widetilde{L(\Sigma)}$, that projects to $\Pi \in L(\Sigma)$. Let $\Lambda^D:[-1,T+1]\to L(\Sigma)$ be the extended Jacobi curve built over the space of piecewise constant variations with discontinuities in $D$ as defined above, and $\tilde{\Lambda}^D(s)$ be the corresponding left-continuous curves in the universal covering with the same initial point $\tilde{\Lambda}^D(-1) = \tilde{\Lambda}_{-1}$, s.t. $\Lambda_{-1} = \Pi$. 

Then there exists a point-wise generalized limit $\tilde{\Lambda}^D(s) \to \tilde{\cL}_s$, such that $\tilde{\cL}_s$ is the lift of the Jacobi curve $\cL_s$ and 
$$
\ind^- \Hess (E_{N_0,t},\nu J_t) [\tilde{\omega},\lambda(t) ] = \frac{1}{2}\left( \Li(\tilde \cL_{T+1},\tilde{\Pi}) - \Li(\tilde \cL_{-1},\tilde{\Pi}) \right) + \dim \left( \bigcap_{s=0}^T \cL_s\cap \Pi \right) - n.
$$
\end{theorem} 

\begin{remark}
The lifts $\tilde{\Lambda}^D$ are constructed using simple monotone curves as described in Remark~\ref{rem:lift}.
\end{remark}

\begin{proof}

\emph{Step 1.} We are going to show that index of the Hessian restricted to the dense sub-space of piecewise constant functions coincides with the index of the Hessian on the whole kernel. This will allow us to apply directly Morse Theorem~\ref{thm:morse_apr}.
 
Map $dE_{N_0,T}$ is a continuous finite rank operator between an infinite dimensional Hilbert manifold that is locally isomoprhic to $T_{\tilde{q}(0)}N_0 \times L^2_k[0,T]$ and $T_{\tilde{q}(0)} M$. We have that the intersection of $\ker d E_{N_0,T} [\tilde{\omega}]$ with the space of piecewise constant functions is dense in $\ker d E_{N_0,T} [\tilde{\omega}]$. Indeed, by continuity of $d E_{N_0,T} [\tilde{\omega}]$ we have that its restriction to the subspace of piecewise constant functions must have the same rank. Therefore the subspace of linear piecewise constant functions splits into two disjoint subspaces: the intersection with the kernel of $d E_{N_0,T} [\tilde{\omega}]$ and a finite-dimensional complement that is isomorphic to its image. If $P_{ker}$ and $P_{fin}$ are two orthogonal projections to these subspace, then given a sequence of piecewise constant function $f_n$ converging to $f\in \ker d E_{N_0,T} [\tilde{\omega}]$, the projections of $P_{fin}f_n$ must converge to zero and $P_{\ker}f_n$ converge to $f$.

At the same time the quadratic form $Q$ is continuous in $T_{\tilde{q}(0)}N_0 \times L^2_k[0,t]$, therefore by restricting to a dense subspace we will get the same index. This implies that we can from the beginning compute the index of $Q$ restricted to the intersection of $\ker d E_{N_0,T}[\tilde{\omega}]$ with piecewise constant functions. 

\emph{Step 2.} We apply Theorem~\ref{thm:morse_apr} to a special sequence of spaces $V_D$. We take a finite number of piecewise constant functions $v_i$, s.t. they span a negative subspace of maximal dimension of the Hessian. Let $D_0$ be a splitting $0 = t_0 < t_1 < ... < t_{N_0} = T$, where $t_i$ are the discontinuity points of $v_i$. Then we can consider any sequence $\{0=t_0^m,...,t_{N_m}^m = T\} = D_m  \supset D_0$, s.t. $\max|t^m_{i+1}-t^m_i| \to 0$ and the corresponding subspace of piecewise constant variations $V^{m} = V^{D_m}$ as in Morse Theorem~\ref{thm:morse_apr}. We also use notations analogous to Theorem~\ref{thm:morse_apr} to define a subspace $V_i^m \subset V^{m}$ of functions that are zero for $t>t_i$ and $(V_i^m)^0 = V_i^m \cap \ker d E_{N_0,T}[\tilde{\omega}]$.

For what follows we will need the following sequence of curves:
$$
\Lambda^m_{t}(s) = 
\begin{cases}
\Pi & \text{ if } -1 \leq s < 0,\\
\Lambda^m(s) & \text{ if } 0\leq s \leq t, \\ 
\Pi & \text{ if } t < s \leq t+1; 
\end{cases}
$$
which are just closed extensions of the restrictions  $\Lambda^m(s)|_{[0,t]}$, where we have shortened the notation for $\Lambda^{D_m}(s)$ just to $\Lambda^{m}(s)$. 

By the additivity Lemma~\ref{lemm:add} and Morse Theorem~\ref{thm:morse_apr} for the piecewise constant approximations, we obtain
\begin{equation}
\label{eq:index}
\ind^- \Hess (E_{N_0,t},\nu J_t)[\tilde{\omega},\lambda(t)] + n = \ind_\Pi \Lambda^m_{t}(s) + \dim \left( \bigcap_{s=-1}^{t} \Lambda^m_{t}(s) \right).
\end{equation}
It only remains to study the limit of the right hand-side when $m\to\infty$.

\emph{Step 3.} We start by considering the second term containing the dimension of the intersections. Since $\cL_{-1}  = \Pi$ and $\Lambda^m_{t}(s) = \Lambda^m(s)$ for $s\in[0,t]$, we have that 
$$
\bigcap_{s=-1}^{t} \Lambda^m_{t}(s)  = \bigcap_{s=0}^{t}(\Pi \cap \Lambda^m(s)).
$$
From the Theorem~\ref{thm:algorithm} and the definition of $\Lambda^m_{t}(s)$ it follows that
$$
\bigcap_{s=0}^{t}(\Pi \cap \Lambda^m(s) ) = \left\{\mu \in T_{\tilde{q}(0)}^\perp N_0 \,:\, \sigma\left( \mu, \int_0^t X(\tau) w(\tau) d\tau \right)=0 \, : \, \forall w(\tau)\in V^m\cap \Omega_{N_0}^t \right\}.
$$
Therefore since $D_i \subset D_{i+1}$, we have that
$$
\bigcap_{s=0}^{t}(\Pi \cap \Lambda^m(s)) \subset \bigcap_{s=0}^{t}(\Pi \cap \Lambda^l(s) ), \qquad \forall l\leq m.
$$
Thus we get a sequence of nested subspace, and since $\Lambda^m(s)$ converge pointwise, this sequence must stabilize for $m$ large enough. 

We claim that
\begin{equation}
\label{eq:one_limit}
\bigcap_{s=0}^{t}(\Pi \cap \cL_s ) = \bigcap_{s=0}^{t}(\Pi \cap \Lambda^m(s) )
\end{equation}
for $m$ large enough. Again, the point-wise convergence implies that
$$
\bigcap_{s=0}^{t}(\Pi \cap \cL_s ) \subset \bigcap_{s=0}^{t}(\Pi \cap \Lambda^m(s)).
$$
The other inclusion holds true as well. Given $\mu \in \bigcap (\Pi \cap \cL_s )$ we can find a sequence
$$
\mu_m \in \bigcap_{s=0}^{t}(\Pi \cap \Lambda^m(s) )
$$
s.t. $\mu_m \to \mu$. But then for any $w \in V^{m}\cap \Omega_{N_0}^t$, we have:
$$
\sigma\left( \mu,\int_0^t X(\tau) w(\tau) d\tau \right) = \lim_{m\to \infty }\left( \mu_m,\int_0^t X(\tau) w(\tau) d\tau \right) = 0.
$$
Thus $\mu \in \Lambda^m(s)$ by definition for $m$ large enough and (\ref{eq:one_limit}) holds.

This way we have shown that
$$
\bigcap_{s=0}^{t}(\cL_s\cap \Pi)  = \bigcap_{s=0}^{t}(\Lambda^m_{t}(s)\cap \Pi), \qquad \forall t\in[0,T].
$$ 

\emph{Step 4.} To arrive at the final result we need to express $\ind_\Pi \Lambda^m(t)$ in terms of the Leray index. 

Fix some $m$ and sequence of Lagrangian planes $\Delta_i^m$, s.t. there exist monotone curves which connect $\Lambda^m_{t}(t_i)$ with $\Lambda^m_{t}(t_{i+1})$ and do not intersect the corresponding $\Delta_i^m$. Let $\Delta$ be any Lagrangian plane. Then from the Maslov index formula in Proposition~\ref{prop:maslov} and the definition of the Leray index we get
\begin{align*}
 \ind_{\Delta} &\Lambda^m_{t}(s) = \frac{1}{2}\sum_{i= -1}^{N_m} \left(   \Ki(\Delta^m_i, \Lambda^m_{t}(t_i), \Delta) - \Ki(\Delta_i^m, \Lambda^m_{t}(t_{i+1}), \Delta)\right) =  \\
 &=-\frac{1}{2}\sum_{i= -1}^{N_m} \left(\Li(\tilde\Lambda^m_{t}(t_{i}),\tilde\Delta) + \Li(\tilde\Delta_i^m, \tilde\Lambda^m_{t}(t_{i})) -\Li(\tilde\Lambda^m_{t}(t_{i+1}),\tilde\Delta) - \Li(\tilde\Delta_i^m, \tilde\Lambda^m_{t}(t_{i+1})) \right).
\end{align*}

By definition Leray index $\Li(\tilde \Lambda_1,\tilde \Lambda_2)$ is locally constant on the set $\{(\tilde{\Lambda}_1,\tilde{\Lambda}_2): \Lambda_1 \cap \Lambda_2 = \{0\} \}$. Since $\Lambda^m_{t}(t_{i})$ and $\Lambda^m_{t}(t_{i+1})$ can be connected by a curve that does not pass through $\Delta_i^m$, we obtain by Lemma~\ref{lemm:leray}
$$
 \Li(\tilde \Delta_i^m, \tilde\Lambda^m_{t}(t_{i})) - \Li(\tilde\Delta_i^m, \tilde\Lambda^m_{t}(t_{i+1})) = 0.
$$
This way we get
$$
\ind_\Delta \Lambda^m_{t}(s) = \frac{1}{2}\left(\Li(\tilde{\Lambda}^m_{t}(t+1),\tilde{\Delta}) - \Li(\tilde{\Lambda}_{-1},\tilde{\Delta})  \right).
$$

\emph{Step 5.} Assume for now that $\tilde{\Lambda}^m(s)$ converges pointwise to a curve $\tilde{\cL}_s$. The previous formula then implies the final result. Indeed, take $t=T$. We put $\Delta = \Pi$ and choose any Lagrangian plane $\Delta'$, s.t. $\Pi\cap \Delta' = \{0\}$ and $\cL_{T+1} \cap \Delta' = \{0\}$. Then by point-wise convergence we also have that $\Lambda^m(T+1) \cap \Delta' = \{0\}$ for $m\in \N$ sufficiently big. Then we obtain from the properties of the Leray index
\begin{align*}
&\ind^- (\Hess (E_{N_0,T},\nu J_T)[\tilde{\omega},\lambda(T)]) + n - \dim \left( \bigcap_{s=0}^{T} \cL_{s}\cap \Pi \right) \\
= &\frac{1}{2}\left( \Li(\tilde{\Lambda}^m(T+1),\tilde{\Pi}) - \Li(\tilde{\Lambda}^m(-1),\tilde{\Pi})\right) = \frac{1}{2}\left( \Li(\tilde{\Lambda}^m(T+1),\tilde{\Delta'})- \right. \\  - & \left. \Li(\tilde{\Lambda}^m(-1),\tilde{\Delta}') + \Ki(\Lambda^m(T+1),\Delta',\Pi) - \Ki(\Lambda^m(-1),\Delta',\Pi) \right).
\end{align*}
By construction $\Lambda^m(T+1) = \Lambda^m(-1) = \Pi$. Therefore the Kashiwara indexes in the expression are zero and we can take limit as $m\to\infty$, since the Leray index is locally constant. Thus we obtain
$$
\ind^- (\Hess (E_{N_0,T},\nu J_T)[\tilde{\omega},\lambda(T)])  = \frac{1}{2}\left( \Li(\tilde{\cL}(T+1),\tilde{\Delta}') - \Li(\tilde{\cL}(-1),\tilde{\Delta'}) \right)- n + \dim \left( \bigcap_{s=0}^{T} \cL_{s}\cap \Pi \right).
$$
Finally we repeat the argument to replace $\Delta'$ back with $\Pi$. We see that the result indeed holds if the pointwise convergence is true. 

\emph{Step 6.} Let us prove the point-wise convergence. Fix a moment of time $t \in [0,T]$. To prove that the sequence $\Lambda^m(t)$ converges point-wise, we are going fix a special Lagrangian plane $\Delta$, s.t. it does not intersect any $\Lambda^m(t)$ or $\cL_t$ and moreover
\begin{equation}
\label{eq:what_we_want}
\Li(\tilde \Lambda^m_{t}(t+1),\tilde \Delta) = \Li(\tilde \Lambda^m(t),\tilde \Delta).
\end{equation}
Then by Lemma~\ref{lem:maslov_ind}, formula (\ref{eq:index}) and steps 3 and 4 we have 
\begin{equation}
\label{eq:one_more_ind}
 \Li(\tilde \Lambda^m(t), \tilde\Delta)  = \Li(\tilde \Lambda_{-1}, \tilde\Delta)- 2\ind^- (\Hess (E_{N_0,t},\nu J_t)[\tilde{\omega},\lambda(t)]) - 2n + 2\dim \left( \bigcap_{s=0}^{t} \cL_{s}\cap \Pi \right)
\end{equation}
for $m$ sufficiently large. Therefore the limit on the left hand side exists. But we recall that by the Theorem~\ref{thm:univer_cover} an open subset in $\{\tilde \Lambda \in \widetilde{L(\Sigma)}\,:\, \Lambda \cap \Delta= \{0\}\}$ can be identified with $\Delta^\pitchfork \times \Z$. And therefore we can take 
$$
\tilde{\cL}_t = \left( \cL_t,\lim_{m\to \infty}\frac{1}{2}\Li(\Lambda^m_t(t),\Delta) \right).
$$

To prove that the Lagrangian plane $\Delta$ with the desired properties exists, we follow until some point the proof of the existence of the $\cL$-derivative. The idea is that we expect that the $\cL$-derivative encodes all the information about the index and the nullity of the Hessian. So we construct a $\cL$-derivative over a finite-dimensional subspace, which contains already the kernel and a negative subspace of maximal dimension. Then adding up variations should not change the $\cL$-derivative to much, at least we can hope that it is not going to produce any contribution to the Maslov index in the process. 

We are going to use the formulas from the definition of a $\cL$-derivative as a linearisation of the Lagrange multiplier rule. We write 
$$
A v = dE_{N_0,t} [\tilde{\omega}](\xi,v).
$$

First of all, we note that directly from the definition, it follows, that variations from $\ker Q \cap \ker A$ do not give any contribution to the $\cL$-derivative. Next we refine our initial partition $D_0$. We assume that $D_0$ is such that the space $V_{D_0}$ is like the subspace $V$ in Lemma~\ref{lem:subspace} with $F = E_{N_0,t}$. Then by Lemma~\ref{lem:ker} for any $D\supset D_0$ one has $\cL_t(V_D)\cap \Pi = \cL_t(V_{D_0})\cap \Pi$. This allows us to search for $\Delta$ in $(\Pi\cap \cL_t(V_{D_0}))^\angle /(\Pi\cap \cL_t(V_{D_0}))$, i.e. we can assume that $\cL_t(V_D) \cap \Pi = \{0\}$. Geometrically this means that we look for $\Delta$ that contains $\cL_t(V_{D_0})\cap \Pi$. Indeed as a result we will get monotone curves that have constant intersection with $\Delta$, so it is going to be enough to replace it with $\Delta^{\Gamma}$, where $\Gamma$ is any isotropic subspace s.t. $\sigma|_{\Gamma\oplus(\Pi\cap \cL_t(V_{D_0}))}$ is symplectic.

We return now to our sequence of partitions $D_m \supset D_0$. All $V^m$ satisfy the assumptions of Theorem~\ref{thm:index_add}. Therefore
$$
\ind_{\Pi}(\Lambda^i(t),\Lambda^j(t)) = 0, \qquad \forall i,j\in \N
$$ 
We note that from the definition of the positive Maslov index it follows, that when $\ind_{\Lambda_1}(\Lambda_2,\Lambda_3)$ is equal to zero or $n$, it implies that $\Lambda_1 \cap \Lambda_2 = \Lambda_1 \cap \Lambda_3 = \{0\}$. So $\Lambda^m(t) \cap \Pi =\{0\}$ for all $m\in \N$.

We choose any $\Delta \in L(\Sigma)$, s.t. $\ind_{\Pi}(\Lambda^0(t),\Delta) = n$. Then by the triangle inequality we get
$$
\ind_{\Pi}(\Lambda^0(t),\Delta) \leq \ind_{\Pi}(\Lambda^0(t),\Lambda^m(t)) + \ind_{\Pi}(\Lambda^m(t),\Delta), \qquad \forall m\in N.
$$
Then $\ind_{\Pi}(\Lambda^m(t),\Delta) = n$. From the relations between the positive Maslov index and the Kashiwara index, we get 
$$
\Ki(\Lambda^m(t), \Pi,\Delta) = 2\ind_{\Pi}(\Lambda^m(t),\Delta) + \dim(\Lambda^m(t) \cap \Delta) - n = n + \dim(\Lambda^m(t) \cap \Delta).
$$
Thus from property 1 of Lemma~\ref{lemm:kashiwara} we obtain
$$
\dim(\Lambda^m(t) \cap \Delta) = 0.
$$
This establishes the existence of $\Delta$. Using the connection between the Kashiwara index and the positive Maslov index together with the antisymmetry of the Kashiwara index, we obtain.
$$
\ind_{\Delta}(\Lambda^m(t),\Pi) = 0.
$$

Therefore all $\Lambda^m(t)$ can be connected to $\Pi$ by a monotone curve that does not intersect $\Delta$. By properties of the Leray index it now follows that (\ref{eq:what_we_want}) holds.

\end{proof}

\section*{Appendices}

\setcounter{section}{0}
\renewcommand{\thesection}{\Alph{section}}

\section{Chronological calculus and second variation}
\label{app:chrono}
In this Appendix we give some basic facts from the chronological calculus and derive a formula for the second variation in the language of symplectic geometry. A more detailed exposition can be found in~\cite{as}.

The idea of the chronological calculus is to reinterpret all geometric objects on a manifold $M$ as linear maps on $C^\infty(M)$. For example, a point $q$ can be seen as a linear operator $\hat{q}: C^\infty(M) \to \R$ defined in a natural way
$$
\hat{q} (a) = a(q), \qquad \forall a \in C^\infty(M).
$$
Similarly one defines an operator analogue of a diffeomorphism $P$:
$$
( \hat q \circ \hat P ) (a) =  P(a(q)), \qquad \forall a \in C^\infty(M), \, \forall q \in M.
$$ 
Here $\hat P : C^\infty(M) \to C^\infty(M)$ is an algebra automorphism, that geometrically is just a change of variables. A vector field $V$ is represented by a differentiation $\hat{V}$ of the algebra $C^\infty(M)$.

In~\cite{as} one can find the proof of the fact, that any algebra homo\-morphism/auto\-mor\-phism/dif\-fe\-ren\-tia\-tion can be represented by a point/dif\-feo\-mor\-phism/vec\-tor field. A one-para\-met\-ric family of these objects can be integrated and differentiated with the usual properties like, for example, the Leibnitz rule.

Consider a non-autonomous vector field $V(t)$ and the corresponding differential equation
$$
\dot{q}(t) = V(t)(q(t))
$$
that can be rewritten in the operator form as
$$
\dot{\hat{q}}(t) = \hat{q}(t) \circ \hat{V}(t).
$$
From here we omit the "hat" in the operator notation, since we will always speak about operators unless it is stated otherwise. If the Cauchy problem for this ODE is well posed, we have a well defined flow $P^t$ that must be a unique solution to the operator equation
\begin{equation}
\label{eq:right}
\dot P^t = P^t \circ V(t).
\end{equation}

A solution to this equation is called the \emph{right chronological exponent} and is denoted by
$$
P^t = \overrightarrow{\exp} \int_0^t V(\tau) d \tau.
$$

Since we know that $P^0 = \id$, we can rewrite equation (\ref{eq:right}) in the integral form
$$
P^t = \id + \int_0^t P^\tau \circ V(\tau) d \tau.
$$
Iterating this expression gives us \emph{the Voltera expansion} for the right chronological exponent
\begin{equation}
\label{eq:voltera}
P^t = \id + \int_0^t V(\tau) d \tau + \int_0^t \int_0^\tau V(\theta) \circ V(\tau) d\theta d\tau  + ...
\end{equation}

The last thing that we need is the variation formulae for the right chronological exponent. Suppose that $V(t), W(t)$ are non-autonomous vector fields and $P^t$ satisfies (\ref{eq:right}). Then the following formulae is true
\begin{align*}
\overrightarrow{\exp}\int_0^t (V(\tau) + W(\tau)) d\tau &= \overrightarrow{\exp}\int_0^t (P^\tau_*)^{-1}W(\tau) d\tau \circ P^t.
\end{align*}
Here $P_*^t$ should be understood as a normal pushforward map, i.e. in the expressions above $(P^t_*)^{-1} W$ should be read as $\widehat{(P^t_*)^{-1} W}$. The proof can be found in the book~\cite{as}.

Now we would like to obtain an explicit expression for the first and second differential of the extended end-point map $\widehat{E}_{N_0,t} = (E_{N_0,t},J_{t})$. 

Using the notations of Section~\ref{sec:optim} by the variation formulae we then find
\begin{align*}
\hat{E}_{N_0,t}(u,\hat{q}(0)) &= \hat q(0) \circ \overrightarrow{\exp}\int_0^t (\hat{P}^\tau_*)^{-1} \left( \hat{f}_{u(\tau)} - \hat{f}_{\tilde{u}(\tau)}\right)d\tau \circ \hat{P}^t =\\
&= \hat q(0) \circ \overrightarrow{\exp}\int_0^t \hat g_{\tau,u(\tau)}d\tau \circ \hat{P}^t.
\end{align*}
Note that $\hat g_{\tau,\tilde{u}(\tau)} \equiv 0$. Using the Voltera expansion~\eqref{eq:voltera} and differentiating w.r.t. to $\hat{q}(0)$ at $(\tilde{q}(0),0) $ and $u(\tau)$ at $\tilde{u}(\tau)$, we obtain for the first variation the following expression
\begin{equation}
\label{eq:first_der}
d\hat{E}_{N_0,t}[\tilde{\omega},\hat{q}(0)](v,\hat{\zeta}) = \hat{P}^t_* \hat{\zeta} + \left( \hat{P}^t_*\int_0^t \left( \left.\frac{\p }{\p u} \right|_{u = \tilde u(\tau) } \hat g_{\tau,u} \right) v(\tau) d\tau \right)(\hat q(t)).
\end{equation}

One can show~\cite{as} that $\lambda(t)$ satisfies the Hamiltonian system 
$$
\dot{\lambda}(t) = \vec{h}(\tilde{u},\lambda(t)), 
$$
where $h$ is the Hamiltonian of PMP defined in Subsection~\ref{sec:optim}. If we restrict the equation (\ref{eq:first_variation}) to $w$ in $L_k^\infty[0,t]$, we obtain
\begin{align*}
0 &= \left \langle \hat \lambda(t), \left( \hat{P}^t_*\int_0^t \left.\frac{\p }{\p u} \right|_{u = \tilde u(\tau) } \hat g_{\tau,u} \cdot v(\tau) d\tau \right)(\hat q(0))\right \rangle = \\
&=\int_0^t \left.\frac{\p }{\p u} \right|_{u = \tilde u(\tau) }\left \langle \hat{\lambda}(t) , \hat{P}^t_* (\hat{P}^\tau_*)^{-1} (\hat{f}_{u} - \hat{f}_{\tilde u(\tau)})  (\hat q(t))  \cdot v(\tau) \right \rangle d\tau = \\
&=\int_0^t \left.\frac{\p h(u,\lambda(t))}{\p u} \right|_{u = \tilde u(\tau) }v(\tau)d\tau.
\end{align*}
Since this equality holds for any $v(t)\in L^\infty_k[0,t]$, we obtain this way the extremality condition~\eqref{eq:first_order}.

Note that since we do not vary the initial value of the functional (i.e.  $J_0 = 0$), we have $\hat{\zeta} = (\zeta,0)\in T_{\tilde{q}(0)}M \times \R$. Thus if we restrict the equation (\ref{eq:first_variation}) to $w\in  T_{\tilde{q}(0)} N_0$, we find
$$
0 = \langle \hat{\lambda}(t), \hat{P}^t_* \hat{\zeta}\rangle = \langle  (\hat{P}^t)^* \hat{\lambda}(t),\hat{\zeta}\rangle  = \langle  \hat{\lambda}(0),\hat{\zeta}\rangle   =\langle \lambda(0),\zeta \rangle.
$$
This way we obtain the transversality conditions~\ref{eq:trans}.

Now we can find an explicit formula for the Hessian, that we use in Subsection~\ref{sec:sympl_morse}. 

\begin{proposition}
The Hessian $\Hess (E_{N_0,t},\nu J_t)[\tilde \omega,\lambda(t)]$ has the following form
\begin{align}
&\Hess (E_{N_0,t},\nu J_t)[\tilde \omega,\lambda(t)]((\zeta_1,v_1),(\zeta_2,v_2)) = \nonumber\\
= &-\int_0^t \sigma \left(\zeta_1 + \int_0^\tau X(\theta) v_1(\theta) d\theta,  X(\tau) v_2(\tau) \right) + b(\tau)(v_1(\tau),v_2(\tau))d\tau. \label{eq:sec_deriv}
\end{align}
\end{proposition}
Note that there is no $\zeta_2$ due to the fact that $(\zeta,v)\in \ker dE_{N_0,T}$ are not independent.

\begin{proof}

We introduce a map $\hat G_{N_0,t} = (\hat P^t)^{-1} \hat{E}_{N_0,t}$. Then we can write  equivalently
$$
Q(v,w) = \langle \hat \lambda(t), d^2\hat{E}_{N_0,t}[\tilde{\omega} ](v,w) \rangle = \langle \hat \lambda(0), d^2\hat{G}_{N_0,t}[\tilde{\omega} ](v,w) \rangle.
$$
To simplify the notations we define
$$
g'(\tau) = \left.\frac{\p}{\p u}\right|_{u = \tilde{u}(\tau)} g_{\tau,u},
$$
and similarly the ``hatted'' $\hat{g}'_{\tau}$ for the extended system. We also define the ``hatted'' version of $X(t)$, which is defined in the same way, using the same Hamiltonian $h(u,\lambda)$, but viewed as a Hamiltonian on $T^*M \times \R^2$ and the corresponding extended Hamiltonian flow $\hat{\Phi}_t$. We note that the projection of $\hat{X}(t)$ to $T(T^*M)$ is exactly $X(t)$ and that the projections of $\hat{X}(t)$ and $X(t)$ to $T M \times \R$ and $TM$ are $\hat{g}'(t)$ and $g'(t)$ for all $t \in [0,T]$. Another important point is that the standard symplectic form $\hat \sigma$ on the extended phase space is equal to
$$
\hat \sigma = \sigma - d\nu \wedge dy,
$$
where $\sigma$ is the standard symplectic form on $T^* M$.

Using the Volterra expansion once more we obtain an explicit formula
\begin{align}
&d^2\hat G_t[\tilde \omega]((\hat \zeta_1,v_1),(\hat \zeta_2,v_2)) = \nonumber \\
= &\int_0^t d_{\hat q} \hat g'(\tau) (\hat \zeta_1 v_2(\tau) +  \hat \zeta_2v_1(\tau))d\tau +  \int_0^t \hat g''(\tau) (v_1(\tau),v_2(\tau)) d\tau + \label{eq:second_var_g}\\
+&\int_0^t \left( \int_0^\tau \hat g'(\theta) v_1(\theta) d\theta \circ \hat g'(\tau) v_2(\tau) + \int_0^\tau \hat g'(\theta) v_2(\theta) d\theta \circ \hat g'(\tau) v_1(\tau) \right) d\tau.\nonumber
\end{align}
By exchanging the order of integration we have
$$
\int_0^t \left(\int_0^\tau \hat g'(\theta) v_2(\theta) d\theta \circ \hat g'(\tau) v_1(\tau) \right) d\tau = \int_0^t \left( \hat g'(\tau) v_2(\tau)  \circ  \int_\tau^t \hat g'(\theta) v_1(\theta)d\theta \right) d\tau.
$$
By adding and subtracting
$$
\int_0^t \hat g'(\tau) v_2(\tau)d\tau \circ \int_0^t \hat g'(\tau) v_1(\tau)d\tau 
$$
we find
\begin{align*}
d^2\hat G_t[\tilde \omega]((\hat \zeta_1,v_1),(\hat \zeta_2,v_2)) &= \int_0^t d_{\hat q} \hat g'(\tau) \zeta_1 v_2(\tau) d\tau + \int_0^t \hat g''(\tau) (v_1(\tau),v_2(\tau)) d\tau + \\
&+ \int_0^t d_{\hat q} \hat g'(\tau)  \hat\zeta_2 v_1(\tau) d\tau + \int_0^t \hat g'(\tau) v_2(\tau)d\tau \circ \int_0^t \hat g'(\tau) v_1(\tau)d\tau + \\
&+\int_0^t \left[ \int_0^\tau \hat g'(\theta) v_1(\theta) d\theta , \hat g'(\tau) v_2(\tau)\right] d\tau .
\end{align*}

We need to reinterpret each summand in terms of sympelctic geometry. One can check (or see~\cite{as}), that 
$$
\int_0^t \langle \hat \lambda(0), \hat g''(\tau)(v_1(\tau),v_2(\tau)) \rangle d\tau = \int_0^t b(\tau)(v_1(\tau),v_2(\tau))d\tau
$$
and 
$$
\left \langle \hat \lambda(0),\int_0^t \left[ \int_0^\tau \hat g'(\theta) v_1(\theta) d\theta , \hat g'(\tau) v_2(\tau)\right]d\tau \right \rangle = \int_0^t \hat \sigma\left( \int_0^\tau \hat X(\theta) v_1(\theta) d\theta , \hat X(\tau) v_2(\tau)\right)d\tau.
$$

To give an interpretation to the first term let us choose Darboux coordinates in $T_{\hat \lambda(0)}(T^*M)$ subordinate to the Lagrangian splitting $T_{\hat \lambda(0)}(T^*(M\times \R))=T_{(\tilde{q}(0),0)}(M\times \R) \times T_{(\tilde{q}(0),0)}^*(M\times \R)$. We note that in this case $-d_{\hat q} \hat g'(\tau) v_2(\tau) \cdot$ can be associated with a covector that is nothing but the projection of $\hat X(\tau) v_2(\tau)$ to the fibre. Therefore we have
$$
\int_0^t d_{\hat q} \hat g'(\tau) \zeta_1 v_2(\tau) d\tau = \int_0^t \hat{\sigma}(\hat \zeta_1, \hat{X}(\tau) v_2(\tau)) d\tau.
$$

Finally for the term in the middle, we use the fact that we are restricting to the kernel of  $dE_{N_0,t}$. On it we have
$$
\zeta_2 = - \int_0^t g'(\tau) v_2(\tau) d\tau.
$$
Since $\hat \zeta_2 = (\zeta_2 ,0)$, we find that
$$
\int_0^t d_{\hat q} \hat g'(\tau) \zeta_1 v_2(\tau) d\tau = \int_0^t g'(\tau) v_2(\tau)d\tau \circ \int_0^t \hat g'(\tau) v_1(\tau)d\tau 
$$
We can write $\hat g'(\tau) = g'(\tau) + (g^0)'(\tau) \p_y $, then 
\begin{align*}
&\int_0^t d_{\hat q} \hat g'(\tau)  \hat\zeta_2 v_1(\tau) d\tau + \int_0^t \hat g'(\tau) v_2(\tau)d\tau \circ \int_0^t \hat g'(\tau) v_1(\tau)d\tau = \\
 = &\int_0^t (g^0)'(\tau)  v_2(\tau)d\tau \,\p_y \circ \int_0^t \hat g'(\tau) v_1(\tau)d\tau,
\end{align*}
but as we have seen the horizontal part is independent of the $y$ variable. So this terms is just zero.

Collecting everything we find an explicit formula for the Hessian
\begin{align}
&\Hess (E_{N_0,t},\nu J_t)[\omega,\lambda(t)]((\zeta_1,v_1),(\zeta_2,v_2)) = \nonumber\\
= &-\int_0^t \hat \sigma \left(\hat \zeta_1 + \int_0^\tau \hat X(\theta) v_1(\theta) d\theta, \hat X(\tau) v_2(\tau) \right) + b(\tau)(v_1(\tau),v_2(\tau))d\tau. \nonumber
\end{align}
This can be simplified even further, if we note that the $\nu$ component of $\hat X(\tau)$ is equal to zero, which can be easily seen from the definitions. Therefore from the explicit form of $\hat{\sigma}$ we derive that
$$
\hat \sigma(\hat X(\theta), \hat X(\tau)) = \sigma(X(\theta), X(\tau)) 
$$
and
$$
\hat \sigma((\zeta_1,0), \hat{X}(\tau)) = \sigma(\zeta_1,X(\tau)),
$$
which proves the proposition.

\end{proof}

\section{Proofs of auxiliary lemmas and Theorem~\ref{thm:main}}

\label{app:existence}

In this appendix we prove some lemmas needed in Subsection~\ref{sec:sympl_morse}. 

We begin with proof of Lemma~\ref{lem:important} and Lemma~\ref{lem:orthogonal}, that we have used several times in the text.

\begin{proof}[Proof of Lemma~\ref{lem:important}]
Assume first that $Q$ is non-degenerate and that $Q|_V$ is non-degenerate as well. Then the statement is reduced to
\begin{equation}
\label{eq:orth1}
\ind^+Q = \ind^+Q|_{V} + \ind^+Q|_{V^\perp},
\end{equation}
which follows from the standard fact from linear algebra that $V\oplus V^\perp = \R^N$~\cite{quadrat}.

Next assume that $Q$ is non-degenerate, but $Q|_{V}$ has a non-trivial kernel. Since under the non-degeneracy assumption $(V^\perp)^\perp = V$, we have 
$$
\ker Q|_V = \ker Q_{V^\perp} = V\cap V^\perp.
$$
Consider to complements $V_1,V_2 \subset V$ which satisfy
\begin{align*}
V_1 \oplus V\cap V^\perp &= V;\\
V_2 \oplus V\cap V^\perp &= V^\perp.
\end{align*}
and form a subspace $U=V_1 \oplus V_2$. Applying formula~\eqref{eq:orth1} to $U$ gives us
$$
\ind^+Q = \ind^+Q|_{U} + \ind^+Q|_{U^\perp}.
$$
By construction
$$
\ind^+Q|_{U} = \ind^+Q|_{V_1} + \ind^+Q|_{V_2} = \ind^+Q|_{V} + \ind^+Q|_{V^\perp}
$$
so it only remains to show $\ind^+Q|_{U^\perp} = \dim V \cap V^\perp$. By non-degeneracy assumption $\dim V + \dim V^\perp = n$. Thus a dimensional count shows
$$
\dim U^\perp = n - \dim V_1 \oplus V_2 = 2 \dim V \cap V^\perp.
$$
But $V \cap V^\perp \subset U^\perp$ is by definition an isotropic subspace inside $U^\perp$. A dimension of a maximal isotropic subspace can not exceed the half of the dimension of the whole space and one has the equality if, and only if, the positive and the negative inertia index of a given quadratic form are equal. Hence
$$
\ind^+ Q|_{U^\perp} = \frac{\dim U^\perp}{2} = \dim V \cap V^\perp.
$$
Thus we obtain
\begin{equation}
\label{eq:orth2}
\ind^+Q = \ind^+Q|_{V} + \ind^+Q|_{V^\perp} + \dim V \cap V^\perp.
\end{equation}

Let us prove the general case. Choose $U\subset \R^N$ such that
$$
U \oplus \ker Q = \R^N.
$$
Then according to the formula~\eqref{eq:orth2} 
$$
\ind^+ Q|_U = \ind^+Q|_{U \cap V} + \ind^+Q|_{(U \cap V)^\perp} + \dim (U \cap V) \cap (U \cap V)^\perp.
$$
where the orthogonal complement is taken with respect to $Q|_U$. Clearly $\ind^+Q|_{U \cap V} = \ind^+Q|_{V}$. For the next term note first that $(U \cap V)^\perp = U \cap V^\perp$. Then as in the first term we obtain $\ind^+Q|_{U \cap V^\perp} = \ind^+Q|_{V^\perp}$. Finally in the last term we have
$$
\dim(U \cap V) \cap (U \cap V)^\perp = \dim U \cap V \cap V^\perp. 
$$
On other side 
$$
\dim (V \cap V^\perp) = \dim (U \cap V \cap V^\perp) + \dim (\ker Q \cap V \cap V^\perp).  
$$
Since $\ker Q \cap V \subset V^\perp$, we have $\ker Q \cap V \cap V^\perp = \ker Q \cap V$. This finishes the proof of Lemma~\ref{lem:important}.

\end{proof}

\begin{proof}[Proof of Lemma~\ref{lem:orthogonal}]
We denote by $W$ the subspace consisting of $v\in V_1^N$ as defined in the statement. By restricting~\eqref{eq:usual} to $w\in V_1^N$, we can easily see that $W \subset (V_1^N)^\perp_Q $. 

To prove the other inclusion we identify the annihilator of $N^\perp$ with the orthogonal compliment of $N$ in $\R^n$. Let us then take a compliment of $A^{-1}(N)$ in $V_2$ which will be isomorphic to $\im A \cap N^\perp $ under $A$ and a basis $e_i$ in this subspace such that the images $Ae_i$ form an orthonormal basis of $\im A \cap N^\perp $. Then if $v \in (V_1^N)^\perp_Q$, by identifying $(\R^n)^*$ with $\R^n$ using the Euclidean inner product, we find that~\eqref{eq:usual} is satisfied if we take
$$
\tilde \xi = -\sum_{i=1}^d \frac{Q(v,e_i)}{|Ae_i|^2}Ae_i,
$$
where $d = \dim (\im A \cap N^\perp)$. Thus $W \supset (V_1^N)^\perp_Q $.

The rest of the statement is proved using exactly the same argument.
\end{proof}

\begin{lemma}
\label{lem:ker}
Let $F: \cU \to M$ be a smooth map from a finite dimensional manifold $\cU$ to a finite dimensional manifold $M$, $J: \cU \to R$ be a smooth functional, and let $(u,\lambda)$ be a Lagrange point of $(F,J)$. Then to any vector in $\cL (F,\nu J)[u,\lambda] \cap \Pi$ we can associate a unique up to an element of $\ker Q \cap \ker dF[u]$ variation $v\in \ker \Hess (F,\nu J)[u,\lambda]$. Consequently
$$
\dim\left(  \ker \Hess (F,\nu J)[u,\lambda] \right) - \dim \left( \ker Q \cap \ker dF[u]  \right)= \dim \left( \cL (F,\nu J)[u,\lambda] \cap \Pi \right).
$$
\end{lemma}

\begin{proof}
The uniqueness part is proved easily. So we can assume that $\ker Q \cap \ker dF[u] = \{0\}$ by factoring out this intersection if necessary. 

If $(\xi,0) \in \cL (F,\nu J)[u,\lambda] \cap \Pi$, then by definition there must exist $v\in \ker dF[u]$, that solves the $\cL$-derivative equation
\begin{equation}
\label{eq:lemm_proof_ker}
\langle\xi, dF[u](w)\rangle + Q(v,w)  = 0, \qquad \forall w\in T_u U.
\end{equation}
Restricting $w$ to $\ker dF[u]$ shows that 
$$\dim\left(  \ker \Hess (F,\nu J)[u,\lambda] \right) \geq \dim \left( \cL (F,\nu J)[u,\lambda] \cap \Pi \right).
$$

Now. Let us assume that $v\in \ker \Hess (F,\nu J)[u,\lambda]$. Then
$$
Q(v,w) = 0, \qquad \forall w \in \ker dF[u].
$$
But this implies that $Q(v,\cdot)$ is linear combination of rows of $dF[u]$. Therefore there must exist $\xi$, s.t. \eqref{eq:lemm_proof_ker} holds.
\end{proof}

\begin{lemma}
\label{lem:subspace}
Let $(u,\lambda)$ be a Lagrangian point of $(F,J): \cU \to M\times \R$. If $\ind^+ Q|_{\ker dF[u]} < \infty$, then there exists a finite-dimensional subspace $V \subset T_u \cU$, s.t. for any $W\supset V$ one has:
\begin{enumerate}
\item $\rank dF[u]|_W = \rank dF[u]|_{V}$,
\item $\ind^+ Q|_W = \ind^+ Q|_V$,
\item $v\in \ker Q|_{W^0}$ and $v \notin V$ iff $v \in \ker Q \cap \ker dF[u]$.
\end{enumerate}
\end{lemma}
The proof is obvious. We just need to take a direct sum of a subspace isomorphic to $\im dF[u]$, a maximal positive subspace and a subspace complementary to $\ker Q \cap \ker dF[u]$.

Using these lemmas we can now give the proof of the following Morse-type result that appeared first in~\cite{agr_lderiv}. 
\begin{theorem}
\label{thm:index_add}
Let $F: \cU \to M$ be a map from a possibly infinite dimensional smooth manifold $\cU$ to a finite dimensional manifold $M$, whose restrictions to finite-dimensional subspaces are smooth, and let $J: \cU \to \R$, $(u,\lambda)$ be a Lagrange point of $(F,J)$. Let $V_1 \subset V_2$ two finite dimensional subspaces of $T_u \cU$ and denote $V_i^0 = V_i \cap \ker dF[u]$. If we choose $V_1,V_2$ to be such that $\rank dF[u]|_{V_1} = \rank dF[u]|_{V_2}$ and $\ker \Hess (F,\nu J)[u,\lambda]|_{V_1^0} = \ker \Hess (F,\nu J)[u,\lambda]|_{V_2^0}$, then
\begin{align}
&\ind^- \Hess (F,\nu J)[u,\lambda]|_{V_2^0} - \ind^- \Hess (F,\nu J)[u,\lambda]|_{V_1^0}  \geq \nonumber \\
\geq &\ind_\Pi\left( \cL (F,\nu J)[u,\lambda](V_1),\cL (F,\nu J)[u,\lambda](V_2) \right). \label{eq:ineq_theorem_long_name}
\end{align}
\end{theorem}

\begin{proof}
From the assumption on the kernels we have by Lemma~\ref{lem:ker}, that
$$
\cL (F,\nu J)[u,\lambda](V_1) \cap \Pi = \cL (F,\nu J)[u,\lambda](V_2) \cap \Pi,
$$
and by Lemma~\ref{lem:important} that 
$$
\ind^+ Q |_{V_2^0} - \ind^+ Q|_{V_1^0} = \ind^+ Q_{(V_1^0)^\perp},
$$ 
where we recall that $Q|_{\ker dF[u]}= - \Hess(F,\nu J)[u,\lambda]$. 

Let us write down explicitly the quadratic form $q$ from the definition of the positive index. Suppose that $(\xi_i,dF[u](v_i))\subset \cL (F,\nu J)[u,\lambda](V_i) $. The quadratic form $q$ is defined on the intersection of $\Pi$ and the sum of $\cL (F,\nu J)[u,\lambda](V_i) $. Thus we define $\xi = \xi_1 + \xi_2$ and assume that $dF[u](v_1+v_2) = 0$. Then we get
$$
q(\xi) = \langle \xi_1, dF[u](v_2)\rangle - \langle\xi_2, dF[u](v_1)\rangle = -\langle\xi_2, d F[u](v_1)\rangle  - \langle\xi_1, d F[u](v_1)\rangle= -\langle\xi, dF[u](v_1)\rangle.
$$

On the other side, Lemma~\ref{lem:orthogonal} tells us that $v\in(V_1^0)^\perp$ if, and only if, there exists $\xi$ in the annihilator of $\in dF[u]$, such that
$$
\langle \xi, dF[u](w) \rangle + Q(v,w) = 0, \qquad \forall w\in V_1.
$$
Then from the definitions of $\cL (F,\nu J)[u,\lambda](V_i) $ it follows that $v = v_1 + v_2$ lies in $(V_1^0)^\perp$ and with covector $\xi = \xi_1 + \xi_2$. So it only remains to check that $Q(v,v)$ is the same as $q(\xi)$. And indeed, we have
$$
Q(v,v) = Q(v_2,v_1+v_2) + Q(v_1+v_2,v_1) = -\langle\xi_2, dF[u](v_1+v_2)\rangle - \langle\xi, dF[u](v_1)\rangle = - \langle\xi, dF[u](v_1)\rangle,
$$
since $v \in \ker dF[u]$ by construction.
\end{proof}

Now we are ready to prove Theorem~\ref{thm:main}. First we show the existence part, since we already have all the necessary results. 

\begin{proof}[Proof of the first part of Theorem~\ref{thm:main}]
For this prove we shorten the notation $\cL (F,\nu J)[u,\lambda](V)$ to $\cL(V)$. Recall that we want to prove existence of the generalised limit that defines the $\cL$-derivative. We will prove that Lagrangian subspaces $\cL(V)$ for different subspaces of variations $V$ all stay in a single coordinate chart and in this chart are represented by symmetric matrices $S_V$ such that $S_V \leq S_W$ whenever $V \subset W$. Thus up to a slight change of the coordinate chart we obtain a bounded monotone increasing generalised sequence and hence there must be a limit.

 Let us take three finite-dimensional subspaces $V \subset V_1 \subset V_2$, where $V$ is like in Lemma~\ref{lem:subspace}. Then as a consequence we have
$$
\cL(V_i) \cap \Pi = \cL(V) \cap \Pi.
$$
Therefore we can assume that $\cL(V_i) \cap \Pi = \{0\}$ by considering $\Sigma = (\cL(V) \cap \Pi)^\angle /(\cL(V) \cap \Pi)$. By the assumptions on $V$ and Theorem~\ref{thm:index_add} we also find that
$$
\ind_\Pi\left( \cL(V_1),\cL(V_2) \right) = 0.
$$

Take any $\Delta \in L(\Sigma)$, s.t. $\ind_\Pi(\cL(V),\Delta) = \dim \Sigma /2$. Then by the triangle inequality and the antisymmetry of the Kashiwara index we have
$$
\ind_{\Pi}(\cL(V_i),\Delta) \geq \ind_{\Pi}(\cL(V),\Delta) - \ind_\Pi(\cL(V),\cL(V_i)) = \frac{\dim \Sigma}{2}.
$$
Therefore $\ind_{\Pi}(\cL(V_i),\Delta) =  \dim \Sigma /2$ and $\cL(V_i) \pitchfork \Delta$.

So we consider a coordinate chart $\Delta^\pitchfork$ centered at $\Pi$. In this chart $\cL(V_i)$ are identified with symmetric operators $S_i: \Pi \to \Delta$. Then for the quadratic form in the definition of $\ind_{\Pi}(\cL(V_i),\Delta)$ we find that
$$
q(p) = \sigma((p,S_ip),(0,-S_i p)) = -p^T S_ip. 
$$
Therefore $S_i$ must be negative-definite symmetric matrices.

We want to show that $\ind_{\Delta}(\cL(V_1),\cL(V_2)) = 0$. Then using the same reasoning we find that $S_2 \geq S_1$, whenever $V_2 \supset V_1$. Thus we obtain a monotone and bounded generalized sequence of symmetric finite-dimensional matrices that must have a limit. In order to compute $\ind_{\Delta}(\cL(V_1),\cL(V_2))$ we use property 4 from Lemma~\ref{lemm:kashiwara} and the cocycle identity. We have that
$$
\Ki(\cL(V_i),\Pi,\Delta) = 2\ind_\Pi(\cL(V_i),\Delta) + \dim(\cL(V_i)\cap \Delta) - \frac{\dim \Sigma}{2} = \frac{\dim \Sigma}{2}.
$$
Then by the cocycle identity we have
\begin{align*}
\Ki(\cL(V_1),\Delta,\cL(V_2)) &= \Ki(\cL(V_1),\Pi,\cL(V_2)) + \Ki(\Delta,\cL(V_2),\Pi) + \Ki(\cL(V_1),\Delta,\Pi) = \\
&=\Ki(\cL(V_1),\Pi,\cL(V_2)).
\end{align*}
Therefore we find that
\begin{align*}
2\ind_\Delta(\cL(V_1),\cL(V_2)) &= \frac{\dim \Sigma}{2} - \dim(\cL(V_1)\cap \cL(V_2)) + \Ki(\cL(V_1),\Delta,\cL(V_2)) =\\
&= \frac{\dim \Sigma}{2} - \dim(\cL(V_1)\cap \cL(V_2)) + \Ki(\cL(V_1),\Pi,\cL(V_2))  =\\
&= 2\ind_\Pi(\cL(V_1),\cL(V_2)) = 0.
\end{align*}
\end{proof}

In order to prove the second part of Theorem~\ref{thm:main}, one can try to show that if $U_0 \subset T_u \cU$ is a dense subset, then
$$
\ind_\Delta (\cL(F,\nu J)[u,\lambda](U_0),\cL(F,\nu J)[u,\lambda]) = 0,
$$
for all $\Delta$ in a dense subset of $L(\Sigma)$. In order to do this we must define $\cL$-derivatives of maps similar to the $\cL$-derivatives of constrained variational problems. Namely if $F:\cU \to M$ is a smooth map, then $(u,\lambda)$ is a Lagrangian point if
$$
\langle \lambda, dF[u](w) \rangle = 0, \qquad \forall w\in T_u \cU
$$
and $\cL$-derivative of this map at $(u,\lambda)$ that we denote by $\cL F[u,\lambda](V)$ consists of vectors $(\xi,dF[u](v))$ which satisfy
$$
\langle \xi, dF[w]\rangle + \langle \lambda, d^2 F[u](v,w)\rangle = 0, \qquad \forall w\in T_u \cU
$$
just like in Definition~\ref{def:l-deriv}. One can easily check that $(u,\lambda)$ is a Lagrange point of the constrained variational problem $(F,J)$ iff $(u,(\lambda,-\nu))$ is a Lagrange point of the extended map $\hat{F} = (F,J)$. Moreover one can easily reconstruct $\cL(F,\nu J)[u,\lambda](V)$ from $\cL \hat F[u,(\lambda,-\nu)](V)$. So we just prove uniqueness of the latter one. To do this we need a lemma similar to Theorem~\ref{thm:index_add}.

\begin{lemma}
\label{lem:maslov_add_2}
Let $F:\cU \to M$ be a smooth map, $a: M \to \R$ be a smooth function and $(u,\lambda)$ a Lagrange point of $F$, s.t. $\lambda = da$. We define a Lagrangian subspace $\Pi_a(\lambda) = T_\lambda(da)$ and the Hessian of $F$ at $u$ to be $Hess(a\circ F)[u] = -d^2(a\circ F)[u]$. Let $V_1\subset V_2$ be two finite-dimensional subspaces, s.t. $\rank dF[u]|_{V_1} = \rank dF[u]|_{V_2}$ and $\ker \Hess(a\circ F)[u]|_{V_1} = \ker \Hess(a\circ F)[u]|_{V_2}$. Then
$$
\ind^- \Hess(a \circ F)[u]|_{V_2} - \ind^- \Hess(a \circ F)[u]|_{V_1} \geq \ind_{\Pi_a(\lambda)}(\cL F[u,\lambda](V_1),\cL F[u,\lambda](V_2)) 
$$
\end{lemma}

\begin{proof}
Although similar to Theorem~\ref{thm:index_add}, the proof of this result is somewhat easier. First of all, like in Lemma~\ref{lem:ker} variations $v \in \ker \Hess(a\circ F)[u]|_{V}$ are in one to one corrispondence with $\cL(F)[u,\lambda](V) \cap \Pi_a(\lambda)$. Indeed, if $v \in \ker \Hess(a\circ F)[u]|_{V}$, then
$$
0 = d^2(a\circ F)(v,w) = d^2 a[F(u)](dF[u](v),dF[u](w)) + d a[F(u)]\circ d^2F[u](v,w), \qquad \forall w\in V.
$$
On the other hand if a vector lies in $\cL(F)[u,\lambda](V) \cap \Pi_a(\lambda)$, then it must be of the form $(d^2 a[F(u)]\circ dF[u](v),dF[u](v))$, where $v$ satisfies the same equation as above. Thus from the assumptions on the ranks and kernels we have that 
$$
\ind^- \Hess(a \circ F)[u]|_{V_2} - \ind^- \Hess(a \circ F)[u]|_{V_1} = \ind^- \Hess(a \circ F)[u]|_{(V_1^\perp)} 
$$
and $\cL(F)[u,\lambda](V_1) \cap \Pi_a(\lambda) =\cL(F)[u,\lambda](V_2) \cap \Pi_a(\lambda)$.

Like in the proof of Theorem~\ref{thm:index_add}, it is enough to show that the quadratic form $q$ from the definition of the positive index can be identified with the restriction of minus $\Hess(a\circ F)[u]_{(V_1)^\perp}$ to a smaller subspace. And indeed, let $(\xi_i,dF[u](v_i))\in \cL F[u,\lambda](V_i)$ be such that $\xi = \xi_1 + \xi_2 = d^2a[F(u)]\circ dF[u](v_1+v_2)$. From the definition of $\cL F[u,\lambda](V_i)$ it follows that in this case $v_1 + v_2 \in (V_1)^\perp$. For the form $q$, using the fact that $\lambda = da$, we find that for $\mu \in \Pi_a(\lambda)\cap(\cL F[u,\lambda](V_1) + \cL F[u,\lambda](V_2))$
\begin{align*}
q(\mu) & = \langle \xi_1, dF[u](v_2)\rangle - \langle \xi_2, dF[u](v_1)\rangle = \langle \xi,dF[u](v_2) \rangle - \langle \xi_2,dF[u](v_1+ v_2) \rangle =  \\
&=  d^2 a[F(u)](dF[u](v_1+v_2),dF[u](v_2) + da[F(u)]\circ d^2F[u](v_1 + v_2,v_2) = \\
&=-\Hess(a\circ F)[u](v_1 + v_2,v_2) = -\Hess(a\circ F)[u](v_1 + v_2,v_1 + v_2),
\end{align*}
which ends the proof.
\end{proof}

\begin{proof}[Proof of the second part of Theorem~\ref{thm:main}]
We again employ the short notation $\cL F[u,\lambda](W) = \cL(W)$. Recall that we want to prove that if $U_0\subset U$ is a dense subspace, then $\cL(U_0) = \cL(U)$. Since in this prove $U$ will denote our reference space of variations we simply write $\cL = \cL(U)$.

Let us assume first that $\cL (U_0) \cap \cL  = \{0\}$. Then we can find $\Delta \in \cL(U_0)^\pitchfork \cap \cL^\pitchfork \cap \Pi^\pitchfork$, s.t. $\ind_\Delta(\cL(U_0),\cL) \neq 0$. But then for any two neighborhoods $O_{\cL(U_0)}$ and $O_{\cL}$ there exist subspaces of variations $V_0\subset U_0$ and $V$, s.t. $\cL(V_0) \in O_{\cL(U_0)}$ and $\cL(V_0 + V) \in O_{\cL}$. By Proposition~\ref{prop:homotopy} we have $\ind_{\Delta}(\cL(V_0), \cL(V_0 + V)) = \ind_{\Delta}(\cL(U_0), \cL)  \neq 0$ for sufficiently small $O_{\cL(U_0)}$ and $O_{\cL}$.
 
Since $\Delta \in \Pi^\pitchfork$, there exists a smooth function $a: M \to R$, s.t. $\Delta = T_{\lambda}(da)$. Therefore if we take $V_0$ to be such that $\ind^- \Hess (a\circ F)|_{V_0}$ is maximal (we can always do this by the definition of a generalized sequence), from Lemma~\ref{lem:maslov_add_2} we obtain that $\ind_{\Delta}(\cL(V_0), \cL(V_0 + V)) = 0 $, which gives us a contradiction. Thus $\cL (U_0) \cap \cL  \neq \{0\}$.

So let us assume that $\cL (U_0) \cap \cL  \neq \{0\}$, but $\cL (U_0) \neq \cL$. We can reduce this case to the one we have just considered. Indeed, let us denote $\Gamma = \cL(U_0) \cap \cL$. Then we can consider the symplectic space $\Gamma^\angle/\Gamma$ and $\cL(U_0)/\Gamma, \cL/\Gamma$ can be identified with two Lagrangian subspaces in this symplectic space. Similarly we can identify $\Delta \cap \Gamma^\angle$, $(\cL(V_0) \cap \Gamma^\angle)/(\cL(V_0) \cap\Gamma)$ and $(\cL(V_0+V) \cap \Gamma^\angle)/(\cL( V_0+V) \cap\Gamma)$ with Lagrangian subspaces in $\Gamma^\angle/\Gamma$. But then we have
\begin{align*}
&\ind_{\Delta}(\cL(V_0),\cL(V_0+ V)) \geq \\
\geq &\ind_{\Delta \cap \Gamma^\angle}((\cL(V_0) \cap \Gamma^\angle)/(\cL(V_0) \cap\Gamma),(\cL(V_0+V) \cap \Gamma^\angle)/(\cL( V_0+V) \cap\Gamma)),
\end{align*}
because it corresponds to the restriction of the form $q$ from the definition of the positive Maslov index to a smaller subspace. Thus again by Lemma~\ref{lem:maslov_add_2} the expression on the right must be zero as well, and we arrive to a contradiction. Therefore  $\cL (U_0) = \cL$.
\end{proof}

\bibliographystyle{plain}
\bibliography{references}

\end{document}